\documentclass[11pt]{amsart}
\usepackage{mathrsfs,amsmath,amssymb,amsthm}
\usepackage[dvipdfm]{graphicx}
\newtheorem{theorem}{Theorem}

\newtheorem{proposition}{Proposition}
\newtheorem{lemma}{Lemma}

\newcommand{\zz}{\mathbf z}
\newcommand{\xx}{\mathbf x}
\newcommand{\yy}{\mathbf y}
\newcommand{\ww}{\mathbf w}

\begin{document}

\title[On deformations of isolated singularities of mixed polynomials]{On deformations of isolated singularities of polar weighted homogeneous mixed polynomials}
\author{Kazumasa Inaba} 
\address{Mathematical Institute, Tohoku University, Sendai 980-8578, Japan}
\email{sb0d02@math.tohoku.ac.jp}


%\texttt{Mathematical Institute, Tohoku University, Sendai 980-8578, Japan} \\
%\texttt{sb0d02@math.tohoku.ac.jp}

%\begin{document}
%\[
%\texttt{Mathematical Institute, Tohoku University, Sendai 980-8578, Japan} 
%\]
%\[
%\texttt{sb0d02@math.tohoku.ac.jp}
%\]

\begin{abstract}
In the present paper, we deform isolated singularities 
of $f\overline{g}$, where $f$ and $g$ are $2$-variable weighted homogeneous complex polynomials,  
%complex polynomials into mixed singularities  
and show that there exists a deformation of $f\overline{g}$  
%the $1$-extension of it is transversal to the submanifold of 
%the bundle of $1$-jets of maps $\Bbb{R}^{2n}$ into $\Bbb{R}^{2}$. 
%Moreover we give deformations of a certain class of complex singularities 
which has only 
indefinite fold singularities and mixed Morse singularities. 
\end{abstract}

\maketitle

%%%%%%%%%%%%%%%%%%%%%%%%%%%%%%%%%%%%%
%%%%%%%%%%%%%%%%%%%%%%%%%%%%%%%%%%%%%
%\pagestyle{plain}

%\tableofcontents
\section{Introduction}
Let $f(\zz)$ be a complex polynomial of variables $\zz = (z_{1}, \dots, z_{n})$. 
A \textit{deformation of $f(\zz)$} is a polynomial mapping 
$F_{t} : \Bbb{C}^{n} \times \Bbb{R} \rightarrow \Bbb{C}, (\zz, t) \mapsto F_{t}(\zz)$, with $F_{0}(\zz) = f(\zz)$. 
Assume that the origin $\mathbf{o}$ is an~isolated singularity of $f(\zz)$. 
For complex singularities, 
it is known that there exist a neighborhood $U$ of the origin and a deformation $F_{t}$ of $f(\zz)$ 
such that $F_{t}(\zz)$ is a complex polynomial and any singularity of $F_{t}(\zz)$ is a Morse singularity in $U$ 
for any $0 < t <<1$ \cite[Chapter 4]{Lg}. 
%an isolated singularity of a complex polynomial can deforme some Morse singularities, 
Here \textit{a Morse singularity} is the singularity of the polynomial map 
$f(\zz) = z_{1}^{2} + \cdots + z_{n}^{2}$ at the origin. 
% \cite[chapter 4]{Lg}. 
%However there exist isolated singularities of mixed polynomials whose homotopy types are different from 
%that of complex polynomials \cite{In1, In2},   
%\textit{The mixed polynomial} is a polynomial of variables 
%$\zz = (z_{1}, \dots, z_{n})$ and $\bar{\zz} = (\bar{z}_{1}, \dots, \bar{z}_{n})$, 
%the terminology of mixed polynomials and proposed a wide class of mixed polynomials which 
%admit the Milnor fibrations, 
%see for instance \cite{O1, O2}. 
Let $\rho_{1}(\xx, \yy)$ and $\rho_{2}(\xx, \yy)$ be real polynomial maps from $\Bbb{R}^{2n}$ to $\Bbb{R}$ 
of variables $\xx = (x_{1}, \dots, x_{n})$ and $\yy = (y_{1}, \dots, y_{n})$. 
Then these real polynomials 
%$\rho_{1}(\xx, \yy)$ and $\rho_{2}(\xx, \yy)$ 
define a polynomial 
of variables $\zz = (z_{1}, \dots, z_{n})$ and $\bar{\zz} = (\bar{z}_{1}, \dots, \bar{z}_{n})$ as 
\[
P(\zz, \bar{\zz}) := \rho_{1}\Bigl(\frac{\zz + \bar{\zz}}{2}, \frac{\zz - \bar{\zz}}{2i}\Bigr) +i\rho_{2}\Bigl(\frac{\zz + \bar{\zz}}{2}, \frac{\zz - \bar{\zz}}{2i}\Bigr), 
\]
where $z_j = x_j + iy_j \ (j =1, \dots, n)$. 
The polynomial $P(\zz, \bar{\zz})$ is called \textit{a mixed polynomial}. 
M. Oka introduced the terminology of mixed polynomials and proposed a wide class of mixed polynomials which 
admit 
Milnor fibrations,
see for instance \cite{O1, O2}. 

Let $\ww$ be an isolated singularity of a mixed polynomial $P(\zz, \bar{\zz})$, $c = P(\ww, \bar{\ww})$ and 
$S_{\ww}^{2n-1}$ be the~$(2n-1)$ - dimensional sphere centered at $\ww$. 
If the link $P^{-1}(c) \cap S_{\ww}^{2n-1}$ is isotopic to the link defined by a complex Morse singularity as an oriented link, 
we say that $\ww$ is \textit{a mixed Morse singularity}. 
In~\cite[Theorem 1]{In1}, 
\cite[Corollary 1, 2]{In2}, 
there exist isolated singularities of mixed polynomials whose homotopy types 
of the vector fields introduced in \cite{NR1} are different from 
those of complex polynomials.   
Thus there exist isolated singularities of real polynomial maps which cannot deform  
mixed  
Morse singularities.  
%If we identity $\Bbb{C}$ with $\Bbb{R}^2$, 
%polynomial maps are regarded as smooth maps from $\Bbb{R}^{2n}$ to $\Bbb{R}^{2}$. 

Let $C^{\infty}(X, Y)$ be the set of smooth maps from $X$ to $Y$, where 
$X$ is a $2n$-dimensional manifold and $Y$ is a $2$-dimensional manifold. 
It is known that the subset of smooth maps from $X$ to $Y$ which have 
only definite fold singularities, indefinite fold singularities or cusps 
is open and dense in $C^{\infty}(X, Y)$ topologized with 
the $C^{\infty}$-topology. 
Moreover definite fold singularities and cusps can be eliminated by homotopy under some conditions 
\cite[Theorem~1,~2]{L1}, 
\cite[Theorem 2.6]{S}. 
In $\dim X = 4$, the fibration with only indefinite fold singularities and Morse singularities is called 
\textit{a broken Lefschetz fibration}, which is recently studied in several papers, 
see for instance \cite[Theorem 1.1]{B}, 
\cite[Theorem 1.1]{G}, (cf. \cite[Theorem~1]{A}). 
We are interested in making deformations of singularities with only indefinite fold singularities and 
mixed Morse singularities. 
This deformed map can be topologically regarded as a~broken Lefschetz fibration. 
If any singularity of a $C^{\infty}$-map $f : X \rightarrow Y$ is an indefinite fold singularity 
or a mixed Morse singularity, 
we call $f$ \textit{a~mixed broken Lefschetz fibration}.

To know if a smooth map $f : X \rightarrow Y$ has only fold singularities or cusps, 
we observe $f$ in the bundle of $r$-jets. 
We introduce the bundle $J^{r}(X, Y)$ of $r$-jets and its submanifolds 
$S_{k}(X, Y)$ and $S_{1}^{2}(X, Y)$ for $k= 1, 2$. 
%We denote the $r$-jet of $f$ at $p$ by $j^{r}f(p)$ and set  
Let $j^{r}f(p)$ be the $r$-jet of $f$ at $p$ and set 
\[
J^{r}(X, Y) := \bigcup_{(p, q) \in X \times Y} J^{r}(X, Y, p, q), 
\]
where $J^{r}(X, Y, p, q) = \{ j^{r}f(p) \mid f(p) = q \}$. 
The set $J^{r}(X, Y)$ is called \textit{the bundle of $r$-jets of maps from $X$ into $Y$}. 
It is known that $J^{r}(X, Y)$ is a~smooth manifold. 
The \textit{$r$-extension} $j^{r}f : X \rightarrow J^{r}(X, Y)$ of $f$ 
is defined by 
$p \mapsto j^{r}f(p)$ where $p \in X$. 
The $1$-jet space is the~$(6n + 2)$-dimensional smooth manifold and 
the $1$-extension $j^{r}f$ of $f$ is a smooth map. 
We define a codimension $(2n - 2 + k)k$-submanifold of $J^{1}(X, Y)$ for $k = 1, 2$ as follows: 
\[
S_{k}(X, Y) = \{ j^{1}f(p) \in J^{1}(X, Y) \mid 
\text{rank}\>df_{p} = 2-k \}. 
\]
%We define subsets $S_{1}(f)$ and $S_{1}^{2}(f)$ of 
%the set of singularities of a smooth map $f: X \rightarrow Y$ as 
%\[
%S_{1}(f) := \{ x \in X \mid \text{the rank of} \ df_{x} = 1 \}, \ \ S_{1}^{2}(f) = S_{1}(f | S_{1}(f)). 
%\] 
%The sub-bundle $S_{1}^{2}(X, Y)$ of $J^{2}(X, Y)$ is defined as follows: 
%the map $f \in~S_{1}^{2}(X, Y)$ if and only if 
%\begin{enumerate}
%\item
%$j^{1}f(p) \in S_{1}(X, Y)$ and $j^{1}f$ is transversal to $S_{1}(X, Y)$ at $p$, 
%\item
%rank$(f | S_{1}(f))(p) = 0$. 
%\end{enumerate}

A smooth map $f : X \rightarrow Y$ is said to be \textit{generic} 
if $f$ satisfies the following conditions: 
\begin{enumerate}
\item
$j^{1}f$ is transversal to $S_{1}(X, Y)$ and $S_{2}(X, Y)$, 
%at $p \in S_{1}(f)$, 
\item
$j^{2}f$ is transversal to $S_{1}^{2}(X, Y)$, 
%at $p \in S_{1}^{2}(f)$, 
\end{enumerate}
where $S_{1}(f) = \{ p \in X \mid \text{rank}\>df_{p} = 1 \}, S_{1}^{2}(f) = S_{1}(f | S_{1}(f))$ and  
%the sub-bundle 
$S_{1}^{2}(X, Y)$ 
%of $J^{2}(X, Y)$ 
is defined as follows: 
%S_{1}^{2}(X, Y) = 
\[
S_{1}^{2}(X, Y) = 
\left\{j^{2}f(p) \in J^{2}(X, Y) \;\left|\; 
\begin{aligned}
& j^{1}f(p) \in S_{1}(X, Y), \\  
& j^{1}f(p)~\text{is transversal to}~S_{1}(X, Y), \\
& \text{rank}(f~|~S_{1}(f))(p)~=~0 \\  
\end{aligned}
\right.\right\}. 
\]
It is well-known that 
a smooth map $f : X \rightarrow Y$ is generic if and only if each singularity of $f$ 
is either a~fold singularity or a cusp.  
Here \textit{a fold singularity} is the singularity of 
$(x_{1}, \dots, x_{2n}) \mapsto (x_{1}, \sum_{j=2}^{2n}\pm x_{j}^{2})$ 
and \textit{a cusp} is the singularity of 
$(x_{1}, \dots, x_{2n}) \mapsto (x_{1}, \sum_{j=3}^{2n}\pm x_{j}^{2} + x_{1}x_{2} + x_{2}^{3})$, 
where $(x_{1}, \dots, x_{2n})$ are the coordinates centered at the singularity. 
%\begin{enumerate}
%\item
%We can choose coordinates $(u, x_{1}, \dots, x_{2n-1})$ centered at $x$ such that $f$ has the form: 
%\[
%\Bigr(u, \sum_{j=1}^{2n-1}\pm x_{j}^{2}\Bigl). 
%\]
%\item
%We can choose coordinates $(u, y, x_{1}, \dots, x_{2n-2})$ centered at $x$ such that $f$ has the form: 
%\[
%\Bigr(u, \sum_{j=1}^{2n-2}\pm x_{j}^{2} + yu + y^{3}\Bigl). 
%\]
%\end{enumerate}
%If the singularity $x$ of $f$ satisfies the condition $(1)$, we call $x$ a \textit{fold} singularity, 
%and if it satisfies the condition $(2)$, a \textit{cusp}. 
If the coefficients of $x_j$ for $j = 2, \dots, 2n$ is 
either all positive or all negative, we say that 
$x$ is a \textit{definite fold} singularity, 
otherwise it is an \textit{indefinite fold} singularity. 
%Thus we see that there exists a deformation $\{ f_{t} \}$ of $f(\zz)$ such that 
%a singularity of $f_{t}(\zz)$ is a fold or cusp for any $0 < t << 1$, where $f(\zz)$ is a complex polynomial. 

Now we state the main theorems. 
Let $f(\zz)$ and $g(\zz)$ be weighted homogeneous complex polynomials. 
Assume $f(\zz)$ and $g(\zz)$ have same weights. 
Then $f(\zz)\overline{g(\zz)}$ satisfies 
\[
f(c \circ \zz)\overline{g(c \circ \zz)} = c^{pq(m-n)}f(\zz)\overline{g(\zz)},  
\]
where $c \circ \zz = (c^qz_{1}, c^{p}z_{2}), c \in \Bbb{C}^{*}$ and 
$pqm$ and $pqn$ are the degrees of the $\Bbb{C}^{*}$-action of $f(\zz)$ and $g(\zz)$ respectively. 
Then we have the Euler equality: 
\[
(pqm)f(\zz) = qz_{1}\frac{\partial f}{\partial z_{1}} + pz_{2}\frac{\partial f}{\partial z_{2}}, \ \
(pqn)g(\zz) = qz_{1}\frac{\partial g}{\partial z_{1}} + pz_{2}\frac{\partial g}{\partial z_{2}}. 
\]
The mixed polynomial $f(\zz)\overline{g}(\zz)$ is \textit{a polar and radial weighted homogeneous mixed polynomial}, 
%\cite{RSV, C, O1, O2}. 
see Section $2.2$ for the definitions. 
Polynomials of this type admit Milnor fibrations \cite{RSV, C, P, O1, O2}. 

We study singularities appearing in a deformation $\{F_{t}\}$ of $f(\zz)\overline{g}(\zz)$  
for any $0 < t << 1$. 
The main theorem is the following. 

\begin{theorem}\label{thm1}
Let $f(\zz)$ and $g(\zz)$ be $2$-variable convenient 
weighted homogeneous complex polynomials such that $f(\zz)\overline{g}(\zz)$ has an isolated singularity at $\mathbf o$ and 
$U$ be a sufficiently small neighborhood of $\mathbf o$. 
Then there exists a deformation $F_{t}(\zz)$ of $f(\zz)\overline{g}(\zz)$ such that 
any singularity of $F_{t}(\zz)$ is 
an indefinite fold singularity
% or the origin $\mathbf o$  
%$j^{1}f_{t}$ is transversal to $S_{1}(\Bbb{R}^{2n}, \Bbb{R}^{2})$ and 
%$S_{2}(\Bbb{R}^{2n}, \Bbb{R}^{2})$ at any singularity of $f_{t}(\zz)$ 
in $U \setminus \{\mathbf{o}\}$ for any $0 < t <<1$. 
\end{theorem}

%The main theorem in this section are Theorem $4$, Theorem $5$ and Theorem $6$ in the introduction. 
%We here recall the statements. 

%\begin{theorem}
%Let $f(\zz)$ be a $2$-variable complex polynomial with an isolated singularity at $\mathbf o$ and 
%let $U$ be a sufficiently small neighborhood of $\mathbf o$. 
%Then there exists a deformation $f_{t}(\zz)$ of $f(\zz)$ such that 
%$j^{1}(f_{t})$ is transversal to $S_{1}(\Bbb{R}^{2n}, \Bbb{R}^{2})$ and 
%$S_{2}(\Bbb{R}^{2n}, \Bbb{R}^{2})$ at any singularity of $f_{t}(\zz)$ in $U$ 
%for any $0 < t <<1$. 
%Let $f(\zz)$ be a $2$-variable complex polynomial with an isolated singularity at $\mathbf o$. 
%Then there exists a deformation $f_{t}(\zz)$ of $f(\zz)$ such that any singularity of $f_{t}(\zz)$ is only
%an indefinite fold singularity for any $0 < t <<1$. 
%\end{theorem}

As an application of Theorem $1$, we show that there exists 
a deformation into mixed broken Lefschetz fibrations. 

\begin{theorem}\label{thm2}
%Let $f(\zz) = z_{1}^{p} + z_{2}^{q}$ be a complex polynomial which has an isolated singularity at the origin $\mathbf{o}$ 
%and let $U$ be a sufficiently small neighborhood of $\mathbf o$. 
Let $F_{t}(\zz)$ be a~deformation of $f(\zz)\overline{g}(\zz)$ in Theorem $1$. 
Then there exists a~deformation $F_{t, s}(\zz)$ of $F_{t}(\zz)$ such that 
$F_{t, s}(\zz)$ is a~mixed broken Lefschetz fibration 
%$f_{t, s}(\zz)$ has only indifinite fold singularities and 
%isolated singularities which define positive Hopf links 
on $U$ where $0 < s <<~t~<<~1$.  
\end{theorem}

This paper is organized as follows. In Section $2$ 
we introduce 
the definition of higher differentials of smooth maps, 
define mixed Hessian $H(P)$ of a mixed polynomial $P(\zz, \bar{\zz})$ 
and show properties of mixed Hessians 
to study singularities of mixed polynomials. 
In Section $3$ and $4$ we prove Theorem $1$ and Theorem $2$ respectively. 

%\subsection*{Acknowledgments}
%\begin{acknowledgment}
The author would like to thank Masaharu Ishikawa, Masayuki Kawashima and 
Nguyen Tat Thang for precious comments. 
%\end{acknowledgment}

\section{Preliminaries}
\subsection{Higher differentials}
In this subsection, we assume that $X$ is an $n$-dimensional manifold and $Y$ is a $2$-dimensional manifold. 
%Let $f, g : X \rightarrow Y$ be smoothmaps from $X$ to $Y$. 
%We consider the necessary and sufficient condition of a fold singularity. 
%We will study when a smooth map $f$ has only fold singularities. 
Let $f : X \rightarrow Y$ be a smooth map and $df : T(X) \rightarrow T(Y)$ be the induced map of $f$, where 
$T(X)$ and $T(Y)$ are the tangent bundles of $X$ and $Y$ respectively. 
If $\Tilde{X}$ is a bundle over $X$ and $\varUpsilon : \Tilde{X} \rightarrow W$ is a map from $\Tilde{X}$ to a space $W$, 
we denote by $X_{x}$ and $\varUpsilon_{x} = \varUpsilon |_{X_{x}}$ the fiber over $x \in X$ and 
the restriction map of $\varUpsilon$ to $X_x$ respectively. 
%Let $X_x$ denote the fiber of $T(X)$ at $x$ and $df_{x} = df|X_{x}$ denote 
%the restriction map of $df$ to $X_x$. 
We set the subset $S_{k}(f)$ of $X$ as 
\[ 
S_{k}(f) = \{ x \in X \mid \text{the rank of} \ df_{x} = 2-k \} \qquad (k = 0, 1, 2). 
\]
Note that $S_{0}(f)$ is the set of regular points of $f$ and 
$S(f) = S_{1}(f) \cup S_{2}(f)$ is the set of singularities of~$f$. 

The following notations are introduced in \cite[Section 2]{L1}. 
Let $U$ and $V$ be small neighborhoods of $x \in X$ and $f(x) \in Y$ such that $f(U) \subset V$. 
Since $T(X)|U$ and $T(Y)|V$ are trivial bundles, 
we can choose bases $\{ u_{i} \}$ and $\{ v_{j} \}$ of the sections of these restricted bundles 
%and denote the pairing of a vector space with its dual by $< , >$. Thus the following relations satisfy 
such that 
\begin{equation*}
\begin{split}
\langle u_{i}(x), u^{*}_{k}(x)\rangle = \delta_{i,k} \ \ \ \text{for all} \  x \in U \\
\langle v_{i}(y), v^{*}_{k}(y)\rangle = \delta_{i,k} \ \ \ \text{for all} \  y \in V, 
\end{split}
\end{equation*}
where $\langle , \rangle$ denotes the pairing of a vector space with its dual, 
$\{ u_{i}^{*} \}$ and $\{ v_{j}^{*} \}$ are dual bases of 
$\{ u_{i} \}$ and $\{ v_{j} \}$ respectively.

%Choose coordinates $\{ \xi_{i} \}$ in $U$ and $\{ \eta_{j} \}$ in $V$. 
%Set $\partial/\partial \xi_{i} = u_{i}, d\xi_{i} = u_{i}^{*},  (\partial/\partial \eta_{j}) \circ f = w_{j}$ and 
%$d\eta_{j} \circ f = w_{j}^{*}$, then $df$ can be represented by 
Choose coordinates $\{ \xi_{i} \}$ in $U$ and $\{ \eta_{j} \}$ in $V$ 
such that $\partial/\partial \xi_{i} = u_{i}, d\xi_{i} = u_{i}^{*},  (\partial/\partial \eta_{j}) = v_{j}$ and 
$d\eta_{j} = v_{j}^{*}$. Then $df$ can be represented by 

\[
df = \sum_{i,j}\frac{\partial (\eta_{j} \circ f)}{\partial \xi_{i}} d\xi_{i} \otimes 
 v_{j}.
\]

Set $E = T(X)|S_{1}(f)$ and $F = T(Y)|f(S_{1}(f))$. Then we can define the following exact sequence 
\[
0 \rightarrow L \rightarrow E \stackrel{df}{\to} F \stackrel{\pi_{1}}{\to} G \rightarrow 0, 
\]
where $L = \ker{df}, G = \text{coker}\>df$ and $\pi$ is the linear map such that $\text{Im}\>\pi = \text{coker}\>{df}$. 

Let $k \in X_x, t \in L_x$ and $a_{i,j} = \partial (\eta_{j} \circ f)/\partial \xi_{i}$. 
We define the map $\varphi^{1} : E \rightarrow L^{*} \otimes F$ by 
\begin{equation*}
\begin{split}
\varphi_{x}^{1}(k, t) &= \sum_{i,j}\Bigl(\langle k, da_{i,j}(x)\rangle \langle t, u_{i}^{*}(x)\rangle \Bigr) v_{j}(x) \\
                   &= \sum_{i,j,m} \Bigl( \langle k, d\xi_{m}(x)\rangle \frac{\partial^{2} (\eta_{j} \circ f)}{\partial \xi_{i} \partial \xi_{m}}
                   \langle t, d\xi_{i}(x)\rangle \Bigr) v_{j}(x)
\end{split}
\end{equation*}
 and then define  the map $d^{2}f : E \rightarrow L^{*} \otimes G$ by 
\[
d^{2}f_{x}(k)(t) = \pi_{1}(\varphi_{x}^{1}(k)(t)). 
\]

By choosing bases of $L_{x}, X_{x}$ and $G_{x}$, the map $d^{2}f$ determines a $n \times (n-1)$ matrix $\phi$. 
$j^{1}f$ is transversal to $S_{1}(X, Y)$ at $S_{1}(f)$ if and only if 
the rank of $\phi$ is equal to $n-1$. 
Moreover the singularity $x \in S_{1}(f)$ is a fold singularity if and only if 
the rank of $\phi$ is equal to $n-1$ and  
the dimension of the kernel of $d^{2}f_{x}$ restricted to $L_{x}$ 
is equal to $0$ \cite{L1}.

\subsection{Polar weighted homogeneous mixed polynomials}
Let $P(\zz, \bar{\zz})$ be a polynomial  
of variables $\zz = (z_1, \dots, z_n)$ and $\bar{\zz} = (\bar{z}_1, \dots, \bar{z}_n)$ given as 
\[
P(\zz, \bar{\zz}) := \sum_{\nu, \mu} c_{\nu, \mu}\zz^{\nu}\bar{\zz}^{\mu}, 
\]
where $\zz^{\nu} = z^{\nu_1}_1 \cdots z^{\nu_n}_n$ for $\nu = (\nu_1, \dots, \nu_n)$ \ 
(respectively $\bar{\zz}^{\mu} = \bar{z}_{1}^{\mu_1} \cdots \bar{z}_{n}^{\mu_n}$ for $\mu = (\mu_1, \dots, \mu_n))$. 
$\bar{z}_j$~represents the complex conjugate of $z_j$. A polynomial $P(\zz, \bar{\zz})$ 
of this form is called a \textit{mixed polynomial} \cite{O1, O2}. 
If $P\bigl((0,\dots,0, z_{j},0,\dots,0), (0,\dots,0, \bar{z}_{j},0,\dots,0)\bigr)$ is non-zero for each $j= 1,\dots,n$, 
then we say that $P(\zz, \bar{\zz})$ is \textit{convenient}. 
%If we identify $\Bbb{C}$ with $\Bbb{R}^{2}$, $f(\zz, \bar{\zz})$ is a smooth map from $\Bbb{R}^{2n}$ to 
%$\Bbb{R}^{2}$.  
A point $\ww \in \Bbb{C}^{n}$ is \textit{a singularity of $P(\zz, \bar{\zz})$} 
if the gradient vectors of $\Re P$ and $\Im P$ are linearly dependent at~$\ww$. 
%To study critical points of a mixed polynomial $f(\zz, \bar{\zz})$, we use the notation: 
%\[
%df(\zz, \bar{\zz}) = \Bigr(\frac{\partial f}{\partial z_{1}}, \dots, \frac{\partial f}{\partial z_{n}}\Bigl), \ \ \
%\bar{d}f(\zz, \bar{\zz}) = \Bigr(\frac{\partial f}{\partial \bar{z}_{1}}, \dots, \frac{\partial f}{\partial \bar{z}_{n}}\Bigl).
%\]
A singularity $\ww$ of $P(\zz, \bar{\zz})$ has the following property.  

\begin{proposition}[\cite{O1} Proposition $1$]
The following conditions are equivalent: 
\begin{enumerate}
\item
$\ww$ is a singularity of $P(\zz, \bar{\zz})$. 
\item
There exists a complex number $\alpha$ with $\lvert \alpha \rvert = 1$ such that 
\[
\Bigr(\overline{\frac{\partial P}{\partial z_{1}}}(\ww), \dots, \overline{\frac{\partial P}{\partial z_{n}}}(\ww)\Bigl) 
= \alpha\Bigr(\frac{\partial P}{\partial \bar{z}_{1}}(\ww), \dots, \frac{\partial P}{\partial \bar{z}_{n}}(\ww)\Bigl). 
\]
\end{enumerate}
\end{proposition}

Let $p_{1}, \dots, p_{n}$ and $q_{1}, \dots, q_{n}$ be integers such that $\gcd(p_1, \dots, p_n) = 1$. 
We define the~$S^1$-action and the~$\Bbb{R}^{*}$-action on $\Bbb{C}^{n}$ as follows: 
\begin{equation*}
\begin{split}
s \circ \zz &= (s^{p_{1}}z_{1}, \dots, s^{p_{n}}z_{n}), \ \ s \in S^{1}. \\
r \circ \zz &= (r^{q_1}z_{1}, \dots, r^{q_n}z_{n}), \ \ r \in \Bbb{R}^{*}. 
\end{split}
\end{equation*}
If there exist  positive integers $d_p$ and $d_{r}$ such that the mixed polynomial $P(\zz, \bar{\zz})$ satisfies 
\begin{equation}%\label{eq1} 
\begin{split}
P(s^{p_{1}}z_{1}, \dots, s^{p_{n}}z_{n}, \bar{s}^{p_1}\bar{z_1}, \dots, \bar{s}^{p_1}\bar{z_n})
 &= s^{d_p}P(\zz, \bar{\zz}), \ \ s \in S^{1},  \notag \\
P(r^{q_1}z_{1}, \dots, r^{q_n}z_{n}, r^{q_1}\bar{z_1}, \dots, r^{q_n}\bar{z_n}) 
&= r^{d_{r}}P(\zz, \bar{\zz}), \ \ r \in \Bbb{R}^{*}, 
\end{split}
\end{equation}
we say that $P(\zz, \bar{\zz})$ is \textit{a polar and radial weighted homogeneous mixed polynomial}. 
If a~polar and radial weighted homogeneous mixed polynomial $P(\zz, \bar{\zz})$ is a complex polynomial, 
we call $P(\zz, \bar{\zz})$ \textit{a weighted homogeneous complex polynomial}. 
Polar and radial weighted homogeneous mixed polynomials  
%a wide class of mixed polynomials which 
admit 
Milnor fibrations,
see for instance \cite{RSV, C, O1, O2}. 
Suppose that $P(\zz, \bar{\zz})$ is a polar and radial weighted homogeneous mixed polynomial. 
Then we have 
\begin{equation*}
\begin{split}
d_{p}P(\zz, \bar{\zz}) &= \textstyle\sum_{j = 1}^{n} p_{j}\Bigl(\frac{\partial P}{\partial z_{j}}z_{j} - \frac{\partial P}{\partial \overline{z}_{j}}\overline{z}_{j} \Bigr), \\
d_{r}P(\zz, \bar{\zz}) &= \textstyle\sum_{j = 1}^{n} q_{j}\Bigl(\frac{\partial P}{\partial z_{j}}z_{j} + \frac{\partial P}{\partial \overline{z}_{j}}\overline{z}_{j} \Bigr). \\
\end{split}
\end{equation*}
If $p_{j} = q_{j}$ for $j = 1, \dots, n$, the above equations give: 
\begin{equation}
\textstyle\sum_{j = 1}^{n}p_{j}\frac{\partial P}{\partial z_{j}}z_{j} = \frac{d_{p} + d_{r}}{2}P(\zz, \bar{\zz}). 
\end{equation}

%We will prove the property of singularities of polar weighted homogeneous mixed polynomials. 
The following claim says that the singularities of $P(\zz, \bar{\zz})$ are 
orbits of the $S^{1}$-action. 
\begin{proposition}
Let $P(\zz, \bar{\zz})$ is a polar weighted homogeneous mixed polynomial. 
If $\ww$ is a~singularity of $P(\zz, \bar{\zz})$, $s \circ \ww$ is also a singularity of $P(\zz, \bar{\zz})$, 
where $s \in S^{1}$. 
\end{proposition}
\begin{proof}
Let $\ww$ be a singularity of a polar weighted homogeneous mixed polynomial $P(\zz, \bar{\zz})$. 
Then there exists $\alpha \in S^{1}$ such that 
\[
\Bigr(\overline{\frac{\partial P}{\partial z_{1}}}(\ww), \dots, \overline{\frac{\partial P}{\partial z_{n}}}(\ww)\Bigl) 
= \alpha\Bigr(\frac{\partial P}{\partial \bar{z}_{1}}(\ww), \dots, \frac{\partial P}{\partial \bar{z}_{n}}(\ww)\Bigl). 
\]
Since $P(\zz, \bar{\zz})$ is a polar weighted homogeneous mixed polynomial, 
$\frac{\partial P}{\partial z_{j}}$ and $\frac{\partial P}{\partial \bar{z}_{j}}$ 
are also. Then we have 
\[
\frac{\partial P}{\partial z_{j}}(s \circ \ww) = s^{d_{p} - p_{j}}\frac{\partial P}{\partial z_{j}}(\ww), \ \ 
\frac{\partial P}{\partial \bar{z}_{j}}(s \circ \ww) = s^{d_{p} + p_{j}}\frac{\partial P}{\partial \bar{z}_{j}}(\ww), 
\]
where $j = 1, \dots, n$ and $s \in S^{1}$. 
So the above equations lead to the following equation: 
\[
\Bigr(\overline{\frac{\partial P}{\partial z_{1}}}(s \circ \ww), \dots, \overline{\frac{\partial P}{\partial z_{n}}}(s \circ \ww)\Bigl) 
= (s^{-2d_{p}}\alpha)\Bigr(\frac{\partial P}{\partial \bar{z}_{1}}(s \circ \ww), \dots, \frac{\partial P}{\partial \bar{z}_{n}}(s \circ \ww)\Bigl). 
\]
Since $\lvert s^{-2d_{p}}\alpha \rvert = 1$, by Proposition $1$, $s \circ \ww$ is also a singularity of $P(\zz, \bar{\zz})$. 
\end{proof}

\subsection{Mixed Hessians}
To study a necessary condition for $P(\zz, \bar{\zz})$ so that the~rank of the~representation matrix of $d^{2}P$ 
is equal to $n-1$, 
%Let $f(\zz, \bar{\zz})$ be a mixed polynomial of variables $\zz = (z_{1}, \dots, z_{n})$ and 
%$\bar{\zz} = (\bar{z}_{1}, \dots, \bar{z}_{n})$. 
we define the matrix $H(P)$ as follows: 
\[ H(P) :=    \left(
        \begin{array}{@{\,}cccccccc@{\,}}
        \Bigl( \frac{\partial^{2}P}{\partial z_{j}\partial z_{k}} \Bigr) & \Bigl( \frac{\partial^{2}P}{\partial z_{j}\partial \bar{z}_{k}} \Bigr) \\ 
        \Bigl( \frac{\partial^{2}P}{\partial \bar{z}_{j}\partial z_{k}} \Bigr) & \Bigl( \frac{\partial^{2}P}{\partial \bar{z}_{j}\partial \bar{z}_{k}} \Bigr)  
        \end{array}
       \right),  \]
where $P(\zz, \bar{\zz})$ is a mixed polynomial. 
We call the matrix $H(P)$ the \textit{mixed Hessian} of $P(\zz, \bar{\zz})$ and show 
some properties of $H(P)$ to study singularities of $P(\zz, \bar{\zz})$. 

The next lemma is useful to understand the mixed Hessian of $P(\zz, \bar{\zz})$. 
\begin{lemma}\label{l51}
Let $A$ and $B$ be $n \times n$ real matrices such that $\det(A + iB) \neq 0$. 
Then there exists a real number $u_0$ such that $\det(A + u_{0}B) \neq 0$. 
%Then the rank of either $A$ or $B$ is equal to $n$. 
\end{lemma}
\begin{proof}
Let $u$ be a complex variable. 
If $B$ is the zero matrix, then $\det(A + uB) = \det(A + iB) \neq 0$. 
Suppose that $B$ is not the zero matrix. 
By the assumption, 
%of $\det(A + iB) \neq 0$
$\det(A + uB)$ is not identically zero. 
Since $\det(A + uB)$ is a polynomial of degree at most $n$, the equation $\det(A + uB) = 0$ has 
finitely many roots. 
Thus there exists a real number $u_0$ which is not a root of $\det(A + uB) = 0$. 
\qedhere
%Let $I$ be the identity matrix, we have 
%\[
%\begin{pmatrix}
%A+iB & 0 \\
%0 & \overline{A+iB} 
%\end{pmatrix}
%\begin{pmatrix}
%I & iI \\
%I & -iI 
%\end{pmatrix} = 
%\begin{pmatrix}
%I & iI \\
%I & -iI 
%\end{pmatrix}
%\begin{pmatrix}
%A & -B \\
%B & A 
%\end{pmatrix}.
%\]
%We calculate the above matrices: 
%\[
%\det(A+iB)\overline{\det(A+iB)} = 
%\det\begin{pmatrix}
%A & -B \\
%B & A 
%\end{pmatrix}.
%\]
%If rank$A$ and rank$B$ are less than $n$, the determinants of the above matrices vanish. 
%This is a contradiction to $\det(A+iB) \neq 0$. 
%Thus the rank of either $A$ or $B$ is equal to $n$. 
\end{proof}

Let $H_{\Bbb{R}}(\eta)$ denote the Hessian of a smooth 
function $\eta : \Bbb{R}^{n} \rightarrow \Bbb{R}$.

\begin{lemma}
Suppose that the rank of $\text{H}(P)$ is $2n$. 
%Then either the rank of $\text{H}_{\Bbb{R}}(\Re f)$ or 
%that of $\text{H}_{\Bbb{R}}(\Im f)$ is greater than $2n-2$. 
By changing the coordinates of $\Bbb{R}^{2}$ if necessary, 
%either the rank of $\text{H}_{\Bbb{R}}(\Re f)$ or 
%that of $\text{H}_{\Bbb{R}}(\Im f)$ is greater than $2n-2$. 
the rank of $H_{\Bbb{R}}(\Im P)$ is $2n$. 
By the same argument, we can also say that by changing the coordinates of $\Bbb{R}^{2}$ if necessary, 
the rank of $H_{\Bbb{R}}(\Re P)$ is $2n$. 
\end{lemma}

\begin{proof}
%By the definition of the differentials of $z_{j}$ and $\bar{z}_{j}$, 
Recall that 
\[
\frac{\partial}{\partial z_{j}} = \frac{1}{2}\Bigl( \frac{\partial}{\partial x_{j}} -i\frac{\partial}{\partial y_{j}} \Bigr), \ \ 
\frac{\partial}{\partial \bar{z}_{j}} = \frac{1}{2}\Bigl( \frac{\partial}{\partial x_{j}} +i\frac{\partial}{\partial y_{j}} \Bigr). 
\]
The second differentials of complex variables can be represented as follows: 
\begin{equation*}
\begin{split}
\frac{\partial^{2}}{\partial z_{j}\partial z_{k}} &= 
\frac{1}{4}\Bigl( \frac{\partial^{2}}{\partial x_{j}\partial x_{k}} - \frac{\partial^{2}}{\partial y_{j}\partial y_{k}} \Bigr) 
-\frac{i}{4}\Bigl( \frac{\partial^{2}}{\partial y_{j}\partial x_{k}} + \frac{\partial^{2}}{\partial x_{j}\partial y_{k}} \Bigr) \\
\frac{\partial^{2}}{\partial z_{j}\partial \bar{z}_{k}} &= 
\frac{1}{4}\Bigl( \frac{\partial^{2}}{\partial x_{j}\partial x_{k}} + \frac{\partial^{2}}{\partial y_{j}\partial y_{k}} \Bigr) 
-\frac{i}{4}\Bigl( \frac{\partial^{2}}{\partial y_{j}\partial x_{k}} - \frac{\partial^{2}}{\partial x_{j}\partial y_{k}} \Bigr) \\
\frac{\partial^{2}}{\partial \bar{z}_{j}\partial z_{k}} &= 
\frac{1}{4}\Bigl( \frac{\partial^{2}}{\partial x_{j}\partial x_{k}} + \frac{\partial^{2}}{\partial y_{j}\partial y_{k}} \Bigr) 
+\frac{i}{4}\Bigl( \frac{\partial^{2}}{\partial y_{j}\partial x_{k}} - \frac{\partial^{2}}{\partial x_{j}\partial y_{k}} \Bigr) \\
\frac{\partial^{2}}{\partial \bar{z}_{j}\partial \bar{z}_{k}} &= 
\frac{1}{4}\Bigl( \frac{\partial^{2}}{\partial x_{j}\partial x_{k}} - \frac{\partial^{2}}{\partial y_{j}\partial y_{k}} \Bigr) 
+\frac{i}{4}\Bigl( \frac{\partial^{2}}{\partial y_{j}\partial x_{k}} + \frac{\partial^{2}}{\partial x_{j}\partial y_{k}} \Bigr). 
\end{split}
\end{equation*}
So the second differentials of a mixed polynomial $P(\zz, \bar{\zz})$ satisfy the following equations:  
\begin{equation*}
\begin{split}
\frac{\partial^{2}P}{\partial z_{j}\partial z_{k}} = 
 &\frac{1}{4}
\Bigl(\frac{\partial^{2}\Re P}{\partial x_{j}\partial x_{k}} + \frac{\partial^{2}\Im P}{\partial x_{j}\partial y_{k}} +
\frac{\partial^{2}\Im P}{\partial y_{j}\partial x_{k}} - \frac{\partial^{2}\Re P}{\partial y_{j}\partial y_{k}}\Bigr) \\
+ &\frac{i}{4}\Bigl(-\frac{\partial^{2}\Re P}{\partial y_{j}\partial x_{k}} - \frac{\partial^{2}\Im P}{\partial y_{j}\partial y_{k}} +
\frac{\partial^{2}\Im P}{\partial x_{j}\partial x_{k}} - \frac{\partial^{2}\Re P}{\partial x_{j}\partial y_{k}}\Bigr) \\
\frac{\partial^{2}P}{\partial z_{j}\partial \bar{z}_{k}} = 
 &\frac{1}{4}
\Bigl(\frac{\partial^{2}\Re P}{\partial x_{j}\partial x_{k}} - \frac{\partial^{2}\Im P}{\partial x_{j}\partial y_{k}} +
\frac{\partial^{2}\Im P}{\partial y_{j}\partial x_{k}} + \frac{\partial^{2}\Re P}{\partial y_{j}\partial y_{k}}\Bigr) \\
+ &\frac{i}{4}\Bigl(-\frac{\partial^{2}\Re P}{\partial y_{j}\partial x_{k}} + \frac{\partial^{2}\Im P}{\partial y_{j}\partial y_{k}} +
\frac{\partial^{2}\Im P}{\partial x_{j}\partial x_{k}} + \frac{\partial^{2}\Re P}{\partial x_{j}\partial y_{k}}\Bigr) \\
\frac{\partial^{2}P}{\partial \bar{z}_{j}\partial z_{k}} = 
 &\frac{1}{4}
\Bigl(\frac{\partial^{2}\Re P}{\partial x_{j}\partial x_{k}} + \frac{\partial^{2}\Im P}{\partial x_{j}\partial y_{k}} -
\frac{\partial^{2}\Im P}{\partial y_{j}\partial x_{k}} + \frac{\partial^{2}\Re P}{\partial y_{j}\partial y_{k}}\Bigr) \\
+ &\frac{i}{4}\Bigl(\frac{\partial^{2}\Re P}{\partial y_{j}\partial x_{k}} + \frac{\partial^{2}\Im P}{\partial y_{j}\partial y_{k}} +
\frac{\partial^{2}\Im P}{\partial x_{j}\partial x_{k}} - \frac{\partial^{2}\Re P}{\partial x_{j}\partial y_{k}}\Bigr) \\
\frac{\partial^{2}P}{\partial \bar{z}_{j}\partial \bar{z}_{k}} = 
 &\frac{1}{4}
\Bigl(\frac{\partial^{2}\Re P}{\partial x_{j}\partial x_{k}} - \frac{\partial^{2}\Im P}{\partial x_{j}\partial y_{k}} -
\frac{\partial^{2}\Im P}{\partial y_{j}\partial x_{k}} - \frac{\partial^{2}\Re P}{\partial y_{j}\partial y_{k}}\Bigr) \\
+ &\frac{i}{4}\Bigl(\frac{\partial^{2}\Re f}{\partial y_{j}\partial x_{k}} - \frac{\partial^{2}\Im P}{\partial y_{j}\partial y_{k}} +
\frac{\partial^{2}\Im P}{\partial x_{j}\partial x_{k}} + \frac{\partial^{2}\Re P}{\partial x_{j}\partial y_{k}}\Bigr). \\
\end{split}
\end{equation*}

The above equations show that the matrix $H_{\Bbb{R}}(\Re P) + iH_{\Bbb{R}}(\Im P)$ has the form: 
\begin{multline*}
 H_{\Bbb{R}}(\Re P) + iH_{\Bbb{R}}(\Im P) =  \\
    \left(
        \begin{array}{@{\,}cccccccc@{\,}}
        \Bigl( \frac{\partial^{2}P}{\partial z_{j}\partial z_{k}} + \frac{\partial^{2}P}{\partial \bar{z}_{j}\partial z_{k}} + \frac{\partial^{2}P}{\partial z_{j}\partial \bar{z}_{k}} + \frac{\partial^{2}P}{\partial \bar{z}_{j}\partial \bar{z}_{k}}\Bigr) 
        & i\Bigl( \frac{\partial^{2}P}{\partial z_{j}\partial z_{k}} + \frac{\partial^{2}P}{\partial \bar{z}_{j}\partial z_{k}} - \frac{\partial^{2}P}{\partial z_{j}\partial \bar{z}_{k}} - \frac{\partial^{2}P}{\partial \bar{z}_{j}\partial \bar{z}_{k}}\Bigr)  \\ 
         i\Bigl( \frac{\partial^{2}P}{\partial z_{j}\partial z_{k}} - \frac{\partial^{2}P}{\partial \bar{z}_{j}\partial z_{k}} + \frac{\partial^{2}P}{\partial z_{j}\partial \bar{z}_{k}} - \frac{\partial^{2}P}{\partial \bar{z}_{j}\partial \bar{z}_{k}}\Bigr)   
        & \Bigl( -\frac{\partial^{2}P}{\partial z_{j}\partial z_{k}} + \frac{\partial^{2}P}{\partial \bar{z}_{j}\partial z_{k}} + \frac{\partial^{2}P}{\partial z_{j}\partial \bar{z}_{k}} - \frac{\partial^{2}P}{\partial \bar{z}_{j}\partial \bar{z}_{k}}\Bigr)   
        \end{array} 
       \right).  
\end{multline*}
We see that $H_{\Bbb{R}}(\Re P) + iH_{\Bbb{R}}(\Im P)$ is congruent to the Hessian $H(P)$. 
Therefore the rank of $H(P)$ is equal to the rank of $H_{\Bbb{R}}(\Re P) + iH_{\Bbb{R}}(\Im P)$. 
We assume that the rank of $H(P)$ is equal to $2n$. 
%Then there exists a $(2n-1) \times (2n-1)$ matrix 
%$H_{\Bbb{R}}'(\Re f) + iH_{\Bbb{R}}'(\Im f)$ such that 
%$\det(H_{\Bbb{R}}'(\Re f) + iH_{\Bbb{R}}'(\Im f)) \neq 0$, 
%where $H_{\Bbb{R}}'(\Re f)$ and $H_{\Bbb{R}}'(\Im f)$ are $(2n-1) \times (2n-1)$ submatrices of 
%$H_{\Bbb{R}}(\Re f)$ and $H_{\Bbb{R}}(\Im f)$ respectively. 
%Assume that, 
By Lemma~\ref{l51}, we can change the coordinates $(w_1, w_2)$ of $\Bbb{R}^{2}$ as 
\[
(w_1, w_2) \mapsto (w_1, w_1 + u_{0}w_{2}) 
\]
such that $u_0$ satisfies $\det(H_{\Bbb{R}}(\Re P) + u_{0}H_{\Bbb{R}}(\Im P)) \neq 0$. 
%So we may assume that the rank of either $H_{\Bbb{R}}'(\Re f)$ or $H_{\Bbb{R}}'(\Im f)$ is equai to $2n-1$. 
%Thus the rank of either $\text{H}_{\Bbb{R}}(\Re f)$ or 
%that of $\text{H}_{\Bbb{R}}(\Im f)$ is greater than $2n-2$. 
With these new coordinates, $P(\zz, \bar{\zz})$ satisfies $\det H_{\Bbb{R}}(\Im P) \neq 0$. 
Thus the rank of $H_{\Bbb{R}}(\Im P))$ is $2n$. 
%We consider the following matrix: 
%\[ A :=    \left(
%        \begin{array}{@{\,}cccccccc@{\,}}
%        H'(\Re f) & -H'(\Im f) \\ 
%        H'(\Im f) & H'(\Re f)  
%        \end{array}
%       \right)  \]
%If the ranks of $H'(\Re f)$ and $H'(\Im f)$ is less than $2n-1$, the determinant $\det(A) = 0$. 
%This is a contradiction to $\det(H'(\Re f) + iH'(\Im f)) \neq 0$. 
\end{proof}

We show a necessary condition of $P(\zz, \bar{\zz})$ so that the rank of the representation matrix of $d^{2}P$ 
is equal to $2n-1$. 

\begin{lemma}
Let $\ww$ belong to $S_{1}(P)$. Suppose the rank of $H(P)$ is $2n$. 
The rank of the~representation matrix of $d^{2}P$ 
is equal to $2n-1$. 
%Then $\zz$ is a fold singularity. 
\end{lemma}

\begin{proof}
Since $\ww \in S_{1}(P)$, one of grad$(\Re P)(\ww)$ and grad$(\Im P)(\ww)$ is non-zero. 
We may assume that grad$(\Re P)(\ww)$ is non-zero. 
% are linearly dependent at~$\zz$. 
By a change of coordinates of $\Bbb{R}^{2}$ as in the proof of Lemma $2$, 
we assume that the rank of $H_{\Bbb{R}}(\Im P)$ is $2n$. 
By change of coordinates of $\Bbb{R}^{2n}$, 
we may further assume that $\partial \Re P / \partial x_{1}(\ww) \neq 0$ and 
write grad$\Im P(\ww) = s \ $grad$\Re P(\ww)$ for some $s \in \Bbb{R}$.  

We then change the coordinates of $\Bbb{R}^{2n}$ as follows: 
\[
  \tilde{x}_{1} = \sum_{\ell = 1}^{2n} \frac{\partial \Re P}{\partial x_{\ell}}(\ww) x_{\ell}, \ \ 
%  + \frac{\partial \Re f}{\partial y_{\ell}}(\zz) y_{\ell}\Bigr), \ 
  \tilde{x}_{j} = x_{j} \ \text{for} \ j \geq 2.
%  , \ \tilde{y}_{k} = y_k \ \text{for} \ k \geq 1. 
\]
By an easy calculus, 
%$\partial \Re f /\tilde{x}_{1} (\zz) = 1$, $\partial \Re f /\tilde{x}_{j} (\zz) = 0$ 
%and $\partial \Re f /\tilde{y}_{k} (\zz) = 0$ where $j \geq 2$ and $k \geq 1$. 
the gradient of $\Re P$ at $\ww$ is equal to $(1, 0,\dots, 0)$.  
%\[     \left(
%        \begin{array}{@{\,}cccccccc@{\,}}
%        1 & 0 & \dots & 0 \\ 
%        0 & 0 & \dots & 0  
%        \end{array}
%       \right).  \]

We define the map $\psi : \Bbb{R}^{2n} \rightarrow \Bbb{R}^{2n}$ by 
\[
 (\tilde{x}_{1}, \dots, \tilde{x}_{2n}) \mapsto (\Re P, \tilde{x}_{2}, \dots, \tilde{x}_{2n}). 
\]
Since the Jacobi matrix of $\psi$ at $\ww$ is the identity matrix, there exists the inverse function $\psi^{-1}$ 
%of $\psi$ on $U$ where $U$ is 
on a~neighborhood of $\ww$. Then the map $(\Re P, \Im P)$ can be represented as follows: 
\begin{equation*}
\begin{split}
(\Re P, \Im P)   &= P(\tilde{x}_{1}, \dots, \tilde{x}_{2n})  \\
                 &= (P \circ \psi^{-1})\circ \psi(\tilde{x}_{1}, \dots, \tilde{x}_{2n}) \\
                 &= (P \circ \psi^{-1})(\Re P, \tilde{x}_{2}, \dots, \tilde{x}_{2n}).                  
\end{split}
\end{equation*}
Let $(x'_{1}, \dots, x'_{2n})$ be the coordinates of $\Bbb{R}^{2n}$ given by 
\[
(x'_{1}, \dots, x'_{2n}) = (\Re P, \tilde{x}_{2}, \dots, \tilde{x}_{2n}). 
\]
Then there exists a map $Q : \Bbb{R}^{2n} \rightarrow \Bbb{R}$ such that 
$P \circ \psi^{-1}(x'_{1}, \dots, y'_{n}) = (x'_{1}, Q(x'_{1}, \dots, x'_{2n}))$.  
%with $(x'_{1}, \dots, x'_{2n}) = (\Re f, \tilde{x}_{2}, \dots, \tilde{x}_{2n})$.  
Since the~singularity $\ww$ belongs to $S_{1}(P)$, the gradient of $Q$ at $\ww$ can be represented by $(s, 0, \dots, 0)$. 
Let $(w_{1}, w_{2})$ be the coordinates of $\Bbb{R}^{2}$. 
If $s$ is not equal to $0$, we change the~coordinates of $\Bbb{R}^{2}$ as 
\[
  \tilde{w}_{1} = w_{1}, \ \tilde{w}_{2} = -sw_{1} + w_{2}, 
\]
so that $P(\zz, \bar{\zz}) = (x'_{1}, -sx'_{1} + Q(x'_{1}, \dots, x'_{2n}))$. 
%Let $\{ \partial/\partial x'_{j}, \partial/\partial y'_{j} \mid j=1, \dots, n \}$ and 
%$\{ \partial/\partial w_{1}, \partial/\partial w_{2} \}$ be basis of $E_{\zz}$ and $F_{f(\zz)}$ respectively. 

%Since the singularity $\zz$ belongs to $S_{1}(P)$, the gradient of $\Im P$ at $\zz$ can be represented by $(s, 0, \dots, 0)$. 
Set $\sum_{j=1}^{2n} a_{j}(\partial/\partial x'_{j}) 
%+ b_{j}(\partial/\partial \tilde{y}_{j}) 
\in~X_{\ww}$, 
then we have 
\begin{equation*}
\begin{split}
dP\biggl(\sum_{j=1}^{2n} \Bigl(a_{j}\frac{\partial}{\partial x'_{j}} 
%+ b_{j}\frac{\partial}{\partial \tilde{y}_{j}}\Bigr) 
\Bigr)\biggr) 
&= \begin{pmatrix} 1 & 0  \dots  0 \end{pmatrix}\begin{pmatrix} a_1 \\ \vdots \\ a_{2n} \end{pmatrix}\frac{\partial}{\partial \tilde{w}_1} 
+ \begin{pmatrix} 0 & 0  \dots  0 \end{pmatrix}\begin{pmatrix} a_1 \\ \vdots \\ a_{2n} \end{pmatrix}\frac{\partial}{\partial \tilde{w}_2} \\
&= a_{1}\frac{\partial}{\partial \tilde{w}_1}.
% + sa_{1}\frac{\partial}{\partial w_2}. 
\end{split}
\end{equation*}
So the kernel $L_{\ww}$ of $dP$ 
is $\{ \sum_{j=2}^{2n}a_{j}(\partial/\partial x'_{j}) 
%+ \sum_{k=1}^{n}b_{k}(\partial/\partial \tilde{y}_{k}) 
\mid a_{j} \in \Bbb{R} \}$ 
and the cokernel $G_{\ww}$ of $dP$ is generated by $\partial/\partial \tilde{w}_2$. 
By the definition of $d^{2}P$, we see that the representation matrix of $d^{2}P$ is 
the Hessian $H_{\Bbb{R}}(Q)$ of $Q$ taking away the first column with these basis. 
%Note that the gradient of $-sx'_{1} + h(x'_{1}, \dots, y'_{n})$ with coordinates $(x'_{1}, \dots, y'_{n})$ is zero. 
%By changing the coordinates, $f(\zz, \bar{\zz})$ can be represented by 
%\[
%(x'_{1}, -sx'_{1} + h(x'_{1}, \dots, y'_{n})). 
%\]
Thus the~rank of the representation matrix of $d^{2}P$ 
is equal to $2n-1$ 
%Then $\zz$ is a fold singularuty of $f(\zz, \bar{\zz})$ 
if and only if 
the rank of the Hessian $H_{\Bbb{R}}(Q)$ of $Q$ taking away the first column is $2n-1$. 
%Since $\zz$ belongs to $S_{1}(f)$, we can assume that the gradient of $\Im f$ at $\zz$ is equal to 
%$(\tilde{s}, 0, \dots, 0)$. 

By the definition of $Q(x'_{1}, \dots, x'_{2n})$,  
$\frac{\partial^{2} Q}{\partial x'_{j}\partial x'_{k}}$ and $\frac{\partial^{2}\Im P}{\partial \tilde{x}_{\ell}\partial \tilde{x}_{m}}$ 
have the following relation: 
\begin{equation*}
\begin{split}
%\frac{\partial^{2}(-sx'_{1} + h)}{\partial x'_{j}\partial x'_{k}} 
%\frac{\partial^{2}h}{\partial x'_{j}\partial x'_{k}} 
%&= 
\frac{\partial^{2} Q}{\partial x'_{j}\partial x'_{k}} 
&= \sum_{\ell, m}\frac{\partial^{2}\Im P}{\partial \tilde{x}_{\ell}\partial \tilde{x}_{m}}
\frac{\partial \tilde{x}_{\ell}}{\partial x'_{j}}\frac{\partial \tilde{x}_{m}}{\partial x'_{k}} 
+ \frac{\partial Q}{\partial \tilde{x}_{\ell}}
\frac{\partial^{2} \tilde{x}_{\ell} }{\partial x'_{j}\partial x'_{k}} \\
&= \sum_{\ell, m}\frac{\partial^{2}\Im P}{\partial \tilde{x}_{\ell}\partial \tilde{x}_{m}} 
\frac{\partial \tilde{x}_{\ell}}{\partial x'_{j}}\frac{\partial \tilde{x}_{m}}{\partial x'_{k}} + 
s\frac{\partial^{2} \tilde{x}_{1} }{\partial x'_{j}\partial x'_{k}}. 
% + \frac{\partial^{2}\Im f}{\partial x_{\ell}\partial y_{m}}
%\frac{\partial x_{\ell}}{\partial \tilde{x}_{j}}\frac{\partial y_{m}}{\partial \tilde{x}_{k}}
%+ \frac{\partial^{2}\Im f}{\partial y_{\ell}\partial y_{m}}
%\frac{\partial y_{\ell}}{\partial \tilde{x}_{j}}\frac{\partial y_{m}}{\partial \tilde{x}_{k}}. 
\end{split}
\end{equation*}
The Hessian $\text{H}_{\Bbb{R}}(Q)$ is equal to 
\[
H_{\Bbb{R}}(\Im P) +  
\biggl(s\Bigl(\frac{\partial^{2}\tilde{x}_{1}}{\partial x'_{1}\partial x'_{1}}\Bigr) \oplus O\biggr), 
\]
where $O$ is the $(2n-1) \times (2n-1)$ zero matrix. 
%Thus the Hessian $\text{H}_{\Bbb{R}}(h)$ is congruent to $\text{H}_{\Bbb{R}}(\Im f)$. 
Therefore the rank of the representation matrix of~$d^{2}P$ 
is equal to $H_{\Bbb{R}}(\Im P)$ taking away the first column. 
%Suppose that 
%grad$(\Re f)$ and 
%grad$(\Im f)$ is not equal to $(0, \dots, 0)$. 
%By changing the coordinates of $\Bbb{R}^{2n}$ linearly, 
%we can assume that the differentials of $\Re f$ and $\Im f$ are non-zero for all variables at first. 
%Then the above construction can apply to any partial differentiation of $\Re f$ and $\Im f$. 
%Now we see, from the definition of $d^{2}f$, that the rank of the representation matrix of $d^{2}f$
%is equal to $2n-1$ 
%if and only if 
%rank$H_{\Bbb{R}}(\Im f) \geq 2n-1$. 
%By Lemma $6$, if the rank of $\text{H}(f)$ is greater than $2n-2$ at $\zz$, 
%the rank of the representation matrix of $d^{2}f$ 
%is equal to $2n-1$. 
Since rank$H_{\Bbb{R}}(\Im P) = 2n$ by the assumption, 
the~rank of the representation matrix of $d^{2}P$ is equal to $2n-1$. 

%$\zz$ is a fold singularity of $f(\zz, \bar{\zz})$. 

We finally consider the case of grad$(\Re P) = (0, \dots, 0)$. 
Since $\ww \in S_{1}(P)$, grad$(\Im P)(\ww)$ is not equal to $(0, \dots, 0)$. 
We change the~coordinates of $\Bbb{R}^{2}$ as 
\[
(w_{1}, w_{2}) \mapsto (w_{1} + w_{2}, -w_{1} + w_{2}). 
\]
Then we have 
\begin{equation*}
\begin{split}
&\Bigl(\bigl(\frac{\partial^{2}\Re P}{\partial x_{j} \partial x_{k}}\bigr) + \bigl(\frac{\partial^{2}\Im P}{\partial x_{j} \partial x_{k}}\bigr)\Bigr) + i 
\Bigl(-\bigl(\frac{\partial^{2}\Re P}{\partial x_{j} \partial x_{k}}\bigr) + \bigl(\frac{\partial^{2}\Im P}{\partial x_{j} \partial x_{k}}\bigr)\Bigr) \\ 
= &(1 - i)\Bigl(\bigl(\frac{\partial^{2}\Re P}{\partial x_{j} \partial x_{k}}\bigr) + i \bigl(\frac{\partial^{2}\Im P}{\partial x_{j} \partial x_{k}}\bigr)\Bigr). 
\end{split}
\end{equation*}
Thus the rank of the Hessian after changing the coordinates of $\Bbb{R}^{2}$ is same that of $H(P)$. 
%If the rank of $\text{H}(f)$ is greater than $2n-2$ at $\zz$, 
%the rank of the representation matrix of $d^{2}f$ 
%is also equal to $2n-1$ 
%in the case of grad$(\Im f) = (0, \dots, 0)$. 
With these new coordinates, grad$(\Re P)$ and grad$(\Im P)$ are not equal to $(0, \dots, 0)$. 
This case had been proved in the previous paragraph. 
\end{proof}

\section{Proof of Theorem~\ref{thm1}}
Let $f(\zz)$ and $g(\zz)$ be 
complex polynomials such that $f(\zz)\overline{g}(\zz)$ has an isolated singularity at the origin. 
We define the $\Bbb{C}^{*}$-action on $\Bbb{C}^{2}$: 
\[
c \circ (z_{1}, z_{2}) := (c^{q}z_{1}, c^{p}z_{2}), \ \ c \in \Bbb{C}^{*}. 
\]
Assume that $f(\zz)$ and $g(\zz)$ are convenient weighted homogeneous complex polynomials, i.e., 
$f(c \circ \zz) = c^{pqm}f(\zz)$ and $g(c \circ \zz) = c^{pqn}g(\zz)$. 
Assume that $m > n$ and $q \geq p$. 
We prepare two lemmas. 
%To prove Theorem $1$, we choose $h(\zz)$ such that the determinant of $H(f\overline{g} + th)$ is not identically equal to $0$. 
%So we will show that 
%$\det H_{\Bbb{C}}(g)(\zz) := \frac{\partial^{2}g}{\partial z_{1} \partial z_{1}}(\zz)\frac{\partial^{2}g}{\partial z_{2} \partial z_{2}}(\zz) - 
%\frac{\partial^{2}g}{\partial z_{1} \partial z_{2}}(\zz)\frac{\partial^{2}g}{\partial z_{2} \partial z_{1}}(\zz) $ or 
%$\det H(f\overline{g})$ is not identically equal to $0$, where $g(\zz)$ is not a linear function. 

\begin{lemma}
Let $g(\zz)$ be a convenient weighted homogeneous complex polynomial 
which has an isolated singularity at the origin. 
Suppose that $g(\zz)$ does not have the following form: 
\[
g(\zz) = \beta_{1} z_{1} + \beta_{2} z_{2}^{k}.
\]
%Then $\frac{\partial^{2}g}{\partial z_{1} \partial z_{1}}(\zz)\frac{\partial^{2}g}{\partial z_{2} \partial z_{2}}(\zz) - 
%\frac{\partial^{2}g}{\partial z_{1} \partial z_{2}}(\zz)\frac{\partial^{2}g}{\partial z_{2} \partial z_{1}}(\zz) \not\equiv 0$. 
%Then by the changing coordinates of $\Bbb{C}^{2}$ if necessary, 
Then $\det H_{\Bbb{C}}(g)(\zz) :=  
\frac{\partial^{2}g}{\partial z_{1} \partial z_{1}}(\zz)\frac{\partial^{2}g}{\partial z_{2} \partial z_{2}}(\zz) - 
\frac{\partial^{2}g}{\partial z_{1} \partial z_{2}}(\zz)\frac{\partial^{2}g}{\partial z_{2} \partial z_{1}}(\zz)$ 
is not identically equal to $0$. 
\end{lemma}
\begin{proof}
Put 
%$g(\zz) = c_{1}z_{1}^{\ell_{1}} + c_{2}z_{1}^{\ell_{2}}z_{2}^{\ell_{3}} + \cdots $, 
$g(\zz) = \sum_{j}c_{j}z_{1}^{\ell_{j}}z_{2}^{k_{j}}$, 
where $\ell_{1} \geq 2, k_{1} = 0$ and $\ell_{j} > \ell_{j'}$ for $j < j'$. 
%Suppose $\ell_{1} \geq 2$. 
We calculate the degrees 
$\deg_{z_{1}} \frac{\partial^{2}g}{\partial z_{1} \partial z_{1}}(\zz)\frac{\partial^{2}g}{\partial z_{2} \partial z_{2}}(\zz)$ and 
$\deg_{z_{1}} \frac{\partial^{2}g}{\partial z_{1} \partial z_{2}}(\zz)\frac{\partial^{2}g}{\partial z_{2} \partial z_{1}}(\zz)$ of $z_{1}$. 
If $k_{2} \geq 2$, 
two degrees are $\ell_{1} + \ell_{2} - 2$ and $2(\ell_{2} - 1)$ respectively. 
Since $\ell_{1}$ is greater than $\ell_{2}$, two degrees are not equal. 
If $k_{2} = 1$, 
%and two degrees are equal, 
by using equation $(1)$, $\ell_{1} = \ell_{2} + (p/q)$. 
If $q$ is greater than $p$, $\ell_{1}$ and $\ell_{2}$ does not satisfy $\ell_{1} = \ell_{2} + (p/q)$. 

So we may assume that $p = q$, $k_{2} = 1$. 
Then $g(\zz)$ has the form: 
\[
g(\zz) = (z_{1} - \tilde{c}z_{2})\tilde{g}(\zz), 
\]
where $\tilde{g}(\zz)$ is a weighted homogeneous polynomial such that 
$\tilde{g}(\zz)$ and $z_{1} - \tilde{c}z_{2}$ have 
no common branches.  
On $\{ (z_{1}, z_{2}) \in \Bbb{C}^{2} \mid z_{1} - \tilde{c}z_{2} = 0\}$, 
the determinant of $H_{\Bbb{C}}(g)(\zz)$ is equal to 
$-\Bigl(\tilde{c}\frac{\partial \tilde{g}}{\partial z_{1}} + \frac{\partial \tilde{g}}{\partial z_{2}}\Bigr)^{2}$. 
If $\det H_{\Bbb{C}}(g)(\zz)$ is identically equal to $0$, the differentials of $\tilde{g}(\zz)$ satisfy 
\[
\tilde{c}\frac{\partial \tilde{g}}{\partial z_{1}} + \frac{\partial \tilde{g}}{\partial z_{2}} = 0. 
\]
Since $\tilde{g}(\zz)$ is a weighted homogeneous polynomial, by using equation $(1)$, $\tilde{g}(\zz)$ 
is equal to 
\[
\frac{1}{n - 1}\Bigl( z_{1}\frac{\partial \tilde{g}}{\partial z_{1}} + z_{2}\frac{\partial \tilde{g}}{\partial z_{2}} \Bigr). 
\]
So $\tilde{g}(\zz)$ vanishes on $\{ (z_{1}, z_{2}) \in \Bbb{C}^{2} \mid z_{1} - \tilde{c}z_{2} = 0\}$. 
%\[
%g(\zz) = c_{1}z_{1}^{\ell_{1}} + c_{2}z_{1}^{\ell_{1} - 1}z_{2} + c_{3}z_{1}^{\ell_{1} - 2}z_{2}^{2} + \cdots + c_{\ell_{1} + 1}z_{2}^{\ell_{1}}. 
%\]
%By using the Euler equality, 
%$n(n - 1)g(\zz)$ is equal to $z_{1}^{2}\frac{\partial^{2}g}{\partial z_{1} \partial z_{1}} + 
%2z_{1}z_{2}\frac{\partial^{2}g}{\partial z_{1} \partial z_{2}} + z_{2}^{2}\frac{\partial^{2}g}{\partial z_{2} \partial z_{2}}$. 
%If $\det H_{\Bbb{C}}(g)(\zz) = \frac{\partial^{2}g}{\partial z_{1} \partial z_{1}}(\zz)\frac{\partial^{2}g}{\partial z_{2} \partial z_{2}}(\zz) - 
%\frac{\partial^{2}g}{\partial z_{1} \partial z_{2}}(\zz)\frac{\partial^{2}g}{\partial z_{2} \partial z_{1}}(\zz) \equiv 0$, 
%we have 
%\[
%n(n - 1)g(\zz) = \biggr(z_{1}\Bigr(\frac{\partial^{2}g}{\partial z_{1} \partial z_{1}}\Bigl)^{\frac{1}{2}} \pm 
%z_{2}\Bigr(\frac{\partial^{2}g}{\partial z_{2} \partial z_{2}}\Bigl)^{\frac{1}{2}} \biggl)^{2}. 
%\] 
%Then $f(\zz)\overline{g}(\zz)$ has singularities on $\{ g(\zz) = 0 \}$. 
Since $\tilde{g}(\zz)$ and $z_{1} - \tilde{c}z_{2}$ have 
no common branches, 
this is a contradiction. 
% of the assumption to $\tilde{g}(\zz)$ and $z_{1} - cz_{2}$. 
Thus $\det H_{\Bbb{C}}(g)(\zz)$ is not identically equal to $0$ for $\ell_{1} \geq 2$. 
\end{proof}

%The determinant of $H(f\overline{g})$ has the following term: 
%\[
%f(\zz)^{2}\frac{\partial^{2}f}{\partial z_{1} \partial z_{1}}(\zz)\frac{\partial^{2}f}{\partial z_{2} \partial z_{2}}(\zz)
%\overline{g(\zz)}^{2}
%\biggl(\overline{ \frac{\partial^{2}g}{\partial z_{1} \partial z_{1}}(\zz)\frac{\partial^{2}g}{\partial z_{2} \partial z_{2}}(\zz) - 
%\frac{\partial^{2}g}{\partial z_{1} \partial z_{2}}(\zz)\frac{\partial^{2}g}{\partial z_{1} \partial z_{2}}(\zz)} \biggr).  
%\] 
%If $\frac{\partial^{2}g}{\partial z_{1} \partial z_{1}}(\zz)\frac{\partial^{2}g}{\partial z_{2} \partial z_{2}}(\zz) \not\equiv 0$ or 
%$\frac{\partial^{2}g}{\partial z_{1} \partial z_{2}}(\zz) \not\equiv 0$, 
%we can deform the coefficinents of $g(\zz)$ such that the above term is not equal to $0$. 
%If $\frac{\partial^{2}g}{\partial z_{1} \partial z_{1}}(\zz)\frac{\partial^{2}g}{\partial z_{2} \partial z_{2}}(\zz)$ and 
%$\frac{\partial^{2}g}{\partial z_{1} \partial z_{2}}(\zz)\frac{\partial^{2}g}{\partial z_{1} \partial z_{2}}(\zz)$ are identically equal to $0$, 
%If $\frac{\partial^{2}g}{\partial z_{1} \partial z_{1}}(\zz)\frac{\partial^{2}g}{\partial z_{2} 
%\partial z_{2}}(\zz) \equiv  
%\frac{\partial^{2}g}{\partial z_{1} \partial z_{2}}(\zz) \equiv 0$, 
%$g(\zz)$ is equal to $z_{1} + \beta z_{2}^{\ell}$ by changing coordinates of $\Bbb{C}^{2}$. 
If $g(\zz)$ has the form $\beta _{1} z_{1} + \beta_{2} z_{2}^{k}$, 
the determinant of $H(f\overline{g})$ is equal to 
\begin{equation*}
\begin{split}
f(\zz)&\biggl(2\frac{\partial^{2}f}{\partial z_{1} \partial z_{2}}(\zz)\frac{\partial f}{\partial z_{1}}(\zz)\frac{\partial f}{\partial z_{2}}(\zz) - 
\frac{\partial^{2}f}{\partial z_{1} \partial z_{1}}(\zz)\Bigr(\frac{\partial f}{\partial z_{2}}(\zz)\Bigr)^{2} - \frac{\partial^{2}f}{\partial z_{2} \partial z_{2}}(\zz)\Bigr(\frac{\partial f}{\partial z_{1}}(\zz)\Bigr)^{2} \biggr) \\
\times &\biggl(\overline{\beta_{1}^{2}g(\zz)\frac{\partial^{2}g}{\partial z_{2} \partial z_{2}}(\zz)}\biggr). 
\end{split}
\end{equation*}
%Suppose that $\ell \geq 2$. 
%We can deform the coefficients of $f(\zz)$ such that the deteminant of $H(f\overline{g})$ is not identically equal to $0$. 
\begin{lemma}
Let $g(\zz) = \beta_{1} z_{1} + \beta_{2} z_{2}^{k}$ with $k \geq 2$. 
%Suppose that $\ell \geq 2$. 
Then the determinant of $H(f\overline{g})$ is not identically equal to $0$. 
\end{lemma}
\begin{proof}
By the assumption, $\frac{\partial^{2}g}{\partial z_{2} \partial z_{2}}(\zz) \not\equiv 0$. 
By using equation $(1)$, the second differentials of $f(\zz)$ satisfy 
\begin{equation*}
\begin{split}
&2\frac{\partial^{2}f}{\partial z_{1} \partial z_{2}}(\zz)\frac{\partial f}{\partial z_{1}}(\zz)\frac{\partial f}{\partial z_{2}}(\zz) - 
\frac{\partial^{2}f}{\partial z_{1} \partial z_{1}}(\zz)\Bigr(\frac{\partial f}{\partial z_{2}}(\zz)\Bigr)^{2} - \frac{\partial^{2}f}{\partial z_{2} \partial z_{2}}(\zz)\Bigr(\frac{\partial f}{\partial z_{1}}(\zz)\Bigr)^{2} \\
&= 
(m - 1)^{2}\biggl( \Bigl(\frac{\partial^{2}f}{\partial z_{1} \partial z_{2}}(\zz)\Bigr)^{2} - \frac{\partial^{2}f}{\partial z_{1} \partial z_{1}}(\zz)\frac{\partial^{2}f}{\partial z_{2} \partial z_{2}}(\zz)\biggr) \\
& \ \ \ \ \ \ \ \ \ \ \ \ \times \biggl(z_{1}^{2}\frac{\partial^{2}f}{\partial z_{1} \partial z_{1}}(\zz) + 
2z_{1}z_{2}\frac{\partial^{2}f}{\partial z_{1} \partial z_{2}}(\zz) + 
z_{2}^{2}\frac{\partial^{2}f}{\partial z_{2} \partial z_{2}}(\zz)  \biggr) \\
&= -(m - 1)^{2}\det H_{\Bbb{C}}(f)\biggl(z_{1}^{2}\frac{\partial^{2}f}{\partial z_{1} \partial z_{1}}(\zz) + 
2z_{1}z_{2}\frac{\partial^{2}f}{\partial z_{1} \partial z_{2}}(\zz) + 
z_{2}^{2}\frac{\partial^{2}f}{\partial z_{2} \partial z_{2}}(\zz)  \biggr). 
\end{split}
\end{equation*}
Since $m > n$, the degree of $f(\zz)$ of $z_{1}$ is greater than $1$. 
By Lemma $4$, $\det H_{\Bbb{C}}(f)(\zz)$ is not identically equal to $0$. 
%Since the polar degree of $f(\zz)$ is greater than that of $g(\zz)$, i., $m > n$,  
This implies $z_{1}^{2}\frac{\partial^{2}f}{\partial z_{1} \partial z_{1}}(\zz) + 
2z_{1}z_{2}\frac{\partial^{2}f}{\partial z_{1} \partial z_{2}}(\zz) + 
z_{2}^{2}\frac{\partial^{2}f}{\partial z_{2} \partial z_{2}}(\zz) \not\equiv 0$. 
Thus the determinant of $H(f\overline{g})$ is not identically equal to $0$. 
\end{proof}

To prove Theorem $1$, we choose $h(\zz)$ such that the determinant of $H(f\overline{g} + th)$ is not identically equal to $0$. 
We divide the proof of Theorem $1$ into two cases: 
\begin{enumerate}
\item
$g(\zz)$ is not a linear function,  
\item
$g(\zz) = \beta_{1} z_{1} + \beta_{2} z_{2}$. 
\end{enumerate}

\subsection{Case $(1)$.}

We define a deformation of $f(\zz)\bar{g}(\zz)$ as follows: 
\[
F_{t}(\zz) = f(\zz)\overline{g(\zz)} + th(\zz),
\]
where $h(\zz) = \gamma_{1} z_{1}^{p(m-n)} + \gamma_{2} z_{2}^{q(m-n)}$ and $0 < t << 1$. 
Let $s$ be a complex number such that $\lvert s \rvert = 1$. 
Then $F_{t}(s \circ \zz)$ satisfies 
\[
F_{t}(s \circ \zz) = f(s \circ \zz)\overline{g}(s \circ \zz) + th(s \circ \zz) = s^{pq(m-n)}F_{t}(\zz). 
\]
So $F_{t}(\zz)$ is also a polar weighted homogeneous mixed polynomial. 
Suppose that $m$ is greater than $n$. 
%We may assume that by defoming of the coefficients of $f(\zz)$ and $g(\zz)$ 
%if necessary, $\det H(F_{0})$ is not equal to $0$.  
%So $\det H(F_{t})$ is also not equal to $0$. 
%The detebminant of $H(f_{t})$ has the following terms: 
%\[
%t\frac{\partial^{2}h}{\partial z_{1} \partial z_{1}}(\zz)f(\zz)\frac{\partial f}{\partial z_{2}}(\zz)
%\overline{\frac{\partial g}{\partial z_{1}}}(\zz)\overline{\frac{\partial^{2}g}{\partial z_{2} \partial z_{2}}}(\zz) + 
%t^{2}\frac{\partial^{2}h}{\partial z_{1} \partial z_{1}}(\zz)\frac{\partial^{2}h}{\partial z_{2} \partial z_{2}}(\zz)
%\Bigr(\overline{\frac{\partial^{2}g}{\partial z_{1} \partial z_{1}}(\zz)\frac{\partial^{2}g}{\partial z_{2} \partial z_{2}}(\zz)
%- \frac{\partial^{2}g}{\partial z_{1} \partial z_{2}}(\zz)\frac{\partial^{2}g}{\partial z_{2} \partial z_{1}}(\zz)}\Bigr).
%\] 
Assume that $f(\zz)\overline{g}(\zz)$ and $h(\zz)$ have no common branches. 
\begin{lemma}
Let $F_{t}(\zz)$ be the above deformation of $f(\zz)\overline{g}(\zz)$. 
If $S_{2}(F_{t}) \neq \emptyset$, $S_{2}(F_{t})$ is only the origin. 
If $S_{2}(F_{t}) = \emptyset$, the origin is a regular point of $F_{t}(\zz)$.
\end{lemma}
\begin{proof}
If $\ww$ belongs to $S_{2}(F_{t})$, by Proposition $1$, 
the singularity $\ww$~satisfies 
$\frac{\partial f}{\partial z_{j}}(\ww)\overline{g}(\ww) + t\frac{\partial h}{\partial z_{j}}(\ww) = 0$ and 
$f(\ww)\overline{\frac{\partial g}{\partial z_{j}}}(\ww) = 0$ for $j = 1, 2$. 
%$\frac{\partial f}{\partial z_{1}}(\ww)\overline{g}(\ww) + t\frac{\partial h}{\partial z_{1}}(\ww) = 0$ for $j = 1, 2$. 
By using equation $(1)$, $f(\zz)\overline{g}(\ww) = 0$ and $th(\ww) = 0$. 
By the~assumption of $h(\ww)$, $\ww$ is equal to the origin. 
Since the origin $\mathbf o$ is an~isolated singularity of $f(\zz)\overline{g}(\zz)$, 
$f(\mathbf{o})\overline{\frac{\partial g}{\partial z_{j}}}(\mathbf{o}) = 0$ for $j = 1, 2$. 
If $S_{2}(F_{t}) = \emptyset$, there exists $j$ such that 
$\frac{\partial f}{\partial z_{j}}(\mathbf{o})\overline{g}(\mathbf{o}) + 
t\frac{\partial h}{\partial z_{j}}(\mathbf{o}) \neq 0$. 
Thus the origin is not a singularity of $F_{t}(\zz)$. 
\end{proof}

Set $f(\zz) = a_{1}z_{1}^{pm} + a_{2}z_{2}^{qm} + z_{1}^{p}z_{2}^{q}f'(\zz)$ 
and $g(\zz) = b_{1}z_{1}^{pm} + b_{2}z_{2}^{qn} + z_{1}^{p}z_{2}^{q}g'(\zz)$, 
where $f'(\zz)$ and $g'(\zz)$ are weighted homogeneous complex polynomials. 

\begin{lemma}\label{l14}
%Let $\gamma_{j}$ be a coefficient of $h(\zz)$ for $j = 1, 2$. 
Suppose that $\gamma_{j}$ is a coefficient of $h(\zz)$ which satisfies 
$\Re \frac{\overline{a_{j}}b_{j}}{\overline{\gamma_{j}}} > 0$ for $j = 1, 2$. 
%There exist coefficients $\gamma_{1}$ and $\gamma_{2}$ of $h(\zz)$ such that 
Then $z_{1}$ and $z_{2}$ are non-zero for any $\ww = (z_{1}, z_{2}) \in S_{1}(F_{t})$. 
\end{lemma}
\begin{proof}
%Set $f(\zz) = c^{1}_{f}z_{1}^{pm} + c^{2}_{f}z_{2}^{qm} + z_{1}^{p}z_{2}^{q}f'(\zz)$ 
%and $g(\zz) = c^{1}_{g}z_{1}^{pm} + c^{2}_{g}z_{2}^{qn} + z_{1}^{p}z_{2}^{q}g'(\zz)$, 
%where $f'(\zz)$ and $g'(\zz)$ are weighted homogeneous complex polynomials. 
Assume that $\ww = (0, z_{2}) \in S_{1}(F_{t})$. 
By Proposition~$1$ and Lemma~$6$, 
$z_{2} \neq 0$ and 
\[
qm\overline{a_{2}}b_{2}z_{2}^{qn}\overline{z}_{2}^{qm-1} + tq(m - n)\overline{\gamma_{2}}\overline{z}_{2}^{q(m -n) - 1} = 
\alpha qna_{2}\overline{b_{2}}z_{2}^{qm}\overline{z}_{2}^{qn - 1}, 
\]
where $\alpha \in S^{1}$. Then we have 
\begin{equation}
m\frac{\overline{a_{2}}b_{2}}{\overline{\gamma_{2}}}z_{2}^{qn}\overline{z}_{2}^{qn} + t(m - n) = 
\alpha n\frac{a_{2}\overline{b_{2}}}{\overline{\gamma_{2}}}z_{2}^{qm}\overline{z}_{2}^{-qm + 2qn}. 
\end{equation}
Since $m$ is greater than $n$ and $\alpha \in S^{1}$, 
the absolute value of $m\frac{\overline{a_{2}}b_{2}}{\overline{\gamma_{2}}}z_{2}^{qn}\overline{z}_{2}^{qn}$ is greater than 
that of $\alpha n\frac{a_{2}\overline{b_{2}}}{\overline{\gamma_{2}}}z_{2}^{qm}\overline{z}_{2}^{-qm + 2qn}$. 
We take $\gamma_{2} \in \Bbb{C}$ which satisfies 
%the real part of 
$\Re \frac{\overline{a_{2}}b_{2}}{\overline{\gamma_{2}}} > 0$. 
% is positive. 
Then $z_{2}$ does not satisfy equation~$(2)$. This is a contradiction. 
Suppose that $\ww = (z_{1}, 0) \in S_{1}(F_{t})$. 
If we take $\gamma_{1} \in \Bbb{C}$ which satisfies 
$\Re \frac{\overline{a_{1}}b_{1}}{\overline{\gamma_{1}}} > 0$, 
the proof is analogous in case~$\ww = (0, z_{2})$. 
Thus we show that coefficients $\gamma_{1}$ and $\gamma_{2}$ of $h(\zz)$ such that 
$z_{1}$ and $z_{2}$ are non-zero for any $\ww = (z_{1}, z_{2}) \in~S_{1}(F_{t})$. 
\end{proof}

We now consider $h(\zz)$ satisfying the following condition: 
\begin{equation}\label{eq1}
\{ \det H(F_{t}) = 0 \} \cap S_{1}(F_{t}) = \emptyset. 
\end{equation}
Note that if $h(\zz)$ satisfies the condition $(3)$, 
the rank of the representation matrix of $d^{2}F_{t}$ is equal to~$3$ by Lemma~$3$.

\begin{lemma}\label{l13}
There exist coefficients $\gamma_{1}$ and $\gamma_{2}$ of $h(\zz)$ 
such that 
%$\gamma_{j}$ satisfies 
$\Re \frac{\overline{a_{j}}b_{j}}{\overline{\gamma_{j}}} > 0$, 
$h(\zz)$ satisfies the condition~$(3)$ and, 
%Recall the calculation of Lemma $7$. 
on~$S_{1}(F_{t})$,  
%the intersection of the following equations is only the origin 
\begin{equation*}
\begin{split}
F_{t}(\zz) = f(\zz)\overline{g}(\zz) + th(\zz) &\neq 0, \\
\end{split}
\end{equation*}
where $j = 1, 2$ and $0 < t <<1$. 
\end{lemma}

\begin{proof}
If $\ww$ is a singularity of $F_{t}(\zz)$, there exists $\alpha \in S^{1}$ such that 
\begin{equation*}
\begin{split}
\overline{\frac{\partial f}{\partial z_{1}}}(\ww)g(\ww) + t\overline{\frac{\partial h}{\partial z_{1}}}(\ww) = 
\alpha f(\ww)\overline{\frac{\partial g}{\partial z_{1}}}(\ww), \\
\overline{\frac{\partial f}{\partial z_{2}}}(\ww)g(\ww) + t\overline{\frac{\partial h}{\partial z_{2}}}(\ww) = 
\alpha f(\ww)\overline{\frac{\partial g}{\partial z_{2}}}(\ww). 
\end{split}
\end{equation*}
So the above equations lead to the following equation:  
\begin{equation*}
\begin{split}
\Phi(\zz, \alpha) :&= \biggl(\overline{\frac{\partial f}{\partial z_{1}}\frac{\partial h}{\partial z_{2}} - \frac{\partial f}{\partial z_{2}}\frac{\partial h}{\partial z_{1}}}\biggr)g(\zz) - 
\alpha \biggl(\overline{\frac{\partial g}{\partial z_{1}}\frac{\partial h}{\partial z_{2}} - \frac{\partial g}{\partial z_{2}}\frac{\partial h}{\partial z_{1}}}\biggr)f(\zz) \\ 
&= \biggl(\overline{\frac{\partial f}{\partial z_{1}}}(\ww)g(\ww) - 
\alpha f(\ww)\overline{\frac{\partial g}{\partial z_{1}}}(\ww)\biggr)\overline{\frac{\partial h}{\partial z_{2}}}(\ww) \\  
&- \biggl(\overline{\frac{\partial f}{\partial z_{2}}}(\ww)g(\ww)    
- f(\ww)\overline{\frac{\partial g}{\partial z_{2}}}(\ww)\biggr)\overline{\frac{\partial h}{\partial z_{1}}}(\ww) \\   
&= 0. 
\end{split}
\end{equation*}

By using equation $(1)$ and Proposition $1$, on $S_{1}(F_{t})$, we have 
\begin{align}
%\begin{split}
pqm\overline{f}(\zz)g(\zz) + pq(m - n)t\overline{h}(\zz) &= \alpha pqnf(\zz)\overline{g}(\zz), \\
t\frac{\partial^{2} h}{\partial z_{j} \partial z_{j}} 
&= t\frac{pq(m-n) - p_{j}}{p_{j}z_{j}}\frac{\partial h}{\partial z_{j}}  \\
&= \frac{pq(m-n) - p_{j}}{p_{j}z_{j}}\Bigr( \overline{\alpha}\overline{f}\frac{\partial g}{\partial z_{j}} - \frac{\partial f}{\partial z_{j}}\overline{g}  \Bigr), \notag
%\end{split}
\end{align}
where $\alpha \in S^{1}$, $p_{1} = q, p_{2} = p$ and $j = 1, 2$. 
By equation $(4)$, $F_{t}(\zz)$ satisfies 
\begin{equation*}
\begin{split}
F_{t}(\zz) &= f(\zz)\overline{g}(\zz) + th(\zz) \\
&= f(\zz)\overline{g}(\zz) + \frac{1}{pq(m - n)}\Bigl( -pqmf(\zz)\overline{g}(\zz) + \overline{\alpha}pqn\overline{f}(\zz)g(\zz) \Bigr) \\
&= f(\zz)\overline{g}(\zz) - \frac{m}{m - n}f(\zz)\overline{g}(\zz) + \frac{n}{m - n}\overline{\alpha}\overline{f}(\zz)g(\zz) \\
&= \frac{-n}{m - n}\Bigl(f(\zz)\overline{g}(\zz) - \overline{\alpha}\overline{f}(\zz)g(\zz)\Bigr). 
\end{split}
\end{equation*}
%$F_{t}(\zz)$ is $\frac{-n}{m - n}\bigl(f(\zz)\overline{g}(\zz) - \overline{\alpha}\overline{f}(\zz)g(\zz)\bigr)$ 
%on $S_{1}(F_{t})$. 
So $F_{t}(\zz)$ vanishes on $S_{1}(F_{t})$ if and only if 
$f(\zz)\overline{g}(\zz) - \overline{\alpha}\overline{f}(\zz)g(\zz)$ is equal to $0$. 
By equation~$(5)$, the Hessian $H(F_{t})$ of~$F_{t}(\zz)$ is equal to 
\[    \left(
        \begin{array}{@{\,}cccccccc@{\,}}
         \omega_{1} %\frac{\partial^{2}f}{\partial z_{1}\partial z_{1}}\overline{g} + 
         %\frac{p(m-n) - 1}{z_{1}}\Bigr( \overline{\alpha}\overline{f}\frac{\partial g}{\partial z_{1}} - \frac{\partial f}{\partial z_{1}}\overline{g}  \Bigr) 
         & \frac{\partial^{2}f}{\partial z_{2}\partial z_{1}}\overline{g} 
         & \frac{\partial f}{\partial z_{1}}\overline{\frac{\partial g}{\partial z_{1}}} 
         & \frac{\partial f}{\partial z_{1}}\overline{\frac{\partial g}{\partial z_{2}}} \\ 
         \frac{\partial^{2}f}{\partial z_{1}\partial z_{2}}\overline{g}  
         & \omega_{2} %\frac{\partial^{2}f}{\partial z_{2}\partial z_{2}}\overline{g} + 
         %\frac{q(m-n) - 1}{z_{2}}\Bigr( \overline{\alpha}\overline{f}\frac{\partial g}{\partial z_{2}} - \frac{\partial f}{\partial z_{2}}\overline{g}  \Bigr)  
         & \frac{\partial f}{\partial z_{2}}\overline{\frac{\partial g}{\partial z_{1}}} 
         & \frac{\partial f}{\partial z_{2}}\overline{\frac{\partial g}{\partial z_{2}}} \\ 
         \frac{\partial f}{\partial z_{1}}\overline{\frac{\partial g}{\partial z_{1}}} 
         & \frac{\partial f}{\partial z_{2}}\overline{\frac{\partial g}{\partial z_{1}}} 
         & f\overline{\frac{\partial^{2}g}{\partial z_{1}\partial z_{1}}} 
         & f\overline{\frac{\partial^{2}g}{\partial z_{2}\partial z_{1}}} \\ 
         \frac{\partial f}{\partial z_{1}}\overline{\frac{\partial g}{\partial z_{2}}} 
         & \frac{\partial f}{\partial z_{2}}\overline{\frac{\partial g}{\partial z_{2}}} 
         & f\overline{\frac{\partial^{2}g}{\partial z_{1}\partial z_{2}}} 
         & f\overline{\frac{\partial^{2}g}{\partial z_{2}\partial z_{2}}}  
        \end{array}
       \right)  \]
at $\ww \in S_{1}(F_{t})$, 
where $\omega_{j} = \frac{\partial^{2}f}{\partial z_{j}\partial z_{j}}\overline{g} + 
\frac{pq(m-n) - p_{j}}{p_{j}z_{j}}\Bigr( \overline{\alpha}\overline{f}\frac{\partial g}{\partial z_{j}} - \frac{\partial f}{\partial z_{j}}\overline{g}  \Bigr)$ for $j = 1, 2$. 
Let $\Psi(\zz, \alpha)$ be the determinant of the above matrix. 
By Lemma $4, 5$, either 
the coefficient of $\overline{\alpha}^{2}$ of $\Psi(\zz, \alpha)$ or $\det H(f\overline{g})$ is non-zero. 
So $\Psi(\zz, \alpha)$ is not identically equal to $0$. 
We define the~$S^1$-action and the $\Bbb{R}^{+}$-action on $\Bbb{C}^{2} \times S^{1}$ as follows: 
\begin{equation*}
\begin{split}
s \circ (z_{1}, z_{2}, \alpha) &:= (s^{q}z_{1}, s^{p}z_{2}, s^{-2d_{f} + 2d_{g}}\alpha), \\
r \circ (z_{1}, z_{2}, \alpha) &:= (r^{q}z_{1}, r^{p}z_{2}, \alpha), 
\end{split}
\end{equation*}
where $s \in S^{1}, r \in \Bbb{R}^{+}$. 
Then $f(\zz)\overline{g}(\zz) - \overline{\alpha}\overline{f}(\zz)g(\zz)$, 
$\Phi(\zz, \alpha)$ and $\Psi(\zz, \alpha)$ are polar and radial weighted homogeneous mixed polynomials. 
Set $V_{1} = \{ (z_{1}, z_{2}, \alpha) \in \Bbb{C}^{2} \times S^{1} \mid 
f(\zz)\overline{g}(\zz) - \overline{\alpha}\overline{f}(\zz)g(\zz) = 0\}$ and  
$V_{2} = \{ (z_{1}, z_{2}, \alpha) \in \Bbb{C}^{2} \times S^{1} \mid \Psi(\zz, \alpha) = 0 \}$. 
Since the dimension of the~algebraic set $V_{j}$ 
%$\{ (z_{1}, z_{2}, \alpha) \in \Bbb{C}^{2} \times S^{1} \mid 
%f(\zz)\overline{g}(\zz) - \overline{\alpha}\overline{f}(\zz)g(\zz) = 0\} \cup 
%\{ (z_{1}, z_{2}, \alpha) \in \Bbb{C}^{2} \times S^{1} \mid \Psi(\zz, \alpha) = 0 \}$ 
is $3$, 
the dimension of the orbit space of~$V_{j}$ 
%$\{ f(\zz)\overline{g}(\zz) - \overline{\alpha}\overline{f}(\zz)g(\zz) = 0\} \cup \{ \Psi(\zz, \alpha) = 0 \}$ 
under the $S^1$-action and the $\Bbb{R}^{+}$-action 
is $1$, for $j = 1, 2$. 
Then the curves $V_1$ and $V_2$ 
%$\{ \Psi(\zz, \alpha) = 0 \}$ 
have finitely many branches which depend only $f(\zz)$ and $g(\zz)$. 

We take a coefficient $\gamma_{j}$ of $h(\zz)$ which satisfies 
$\Re \frac{\overline{a_{j}}b_{j}}{\overline{\gamma_{j}}} > 0$ for $j = 1, 2$. 
Assume that $\ww = (z_{1}, z_{2}) \in~S_{1}(F_{t})$ satisfies 
\begin{equation*}
\begin{split}
\overline{\frac{\partial f}{\partial z_{1}}}(\ww)g(\ww)\overline{\frac{\partial h}{\partial z_{2}}}(\ww) &=  
f(\ww)\overline{\frac{\partial g}{\partial z_{1}}}(\ww)\overline{\frac{\partial h}{\partial z_{2}}}(\ww) = 0,   \\
\overline{\frac{\partial f}{\partial z_{2}}}(\ww)g(\ww)\overline{\frac{\partial h}{\partial z_{1}}}(\ww) &=  
f(\ww)\overline{\frac{\partial g}{\partial z_{2}}}(\ww)\overline{\frac{\partial h}{\partial z_{1}}}(\ww) = 0.   
\end{split}
\end{equation*}
If $\ww$ satisfies $\overline{\frac{\partial f}{\partial z_{j}}}(\ww)g(\ww) = 
f(\ww)\overline{\frac{\partial g}{\partial z_{j}}}(\ww) = 0$, 
by Proposition $1$, 
$\frac{\partial h}{\partial z_{j}}(\ww) = 0$, where $j = 1$ or $j = 2$.  
Since $h(\zz)$ is $\gamma_{1} z_{1}^{p(m-n)} + \gamma_{2} z_{2}^{q(m-n)}$, either $z_{1}$ or $z_{2}$ is $0$. 
By Lemma \ref{l14}, this is a~contradiction. 
Hence $\Phi(\zz, \alpha)$ 
is non-zero on~$S_{1}(F_{t})$. 

Since the curves $V_1$ and $V_2$ 
%$\{ f(\zz)\overline{g}(\zz) - \overline{\alpha}\overline{f}(\zz)g(\zz) = 0\}$ and 
%$\{ \Psi(\zz, \alpha) = 0 \}$ 
have finitely many branches, 
%setting $f(\zz) = a_{1}z_{1}^{pm} + a_{2}z_{2}^{qm} + z_{1}^{p}z_{2}^{q}f'(\zz)$ 
%and $g(\zz) = b_{1}z_{1}^{pm} + b_{2}z_{2}^{qn} + z_{1}^{p}z_{2}^{q}g'(\zz)$, 
we can choose coefficients $\gamma_{1}$ and $\gamma_{2}$ of $h(\zz)$ such that 
the intersection of 
$(V_{1} \cup V_{2})$ 
and $\{ \Phi(\zz, \alpha) = 0 \}$ is only the origin and 
$\Re \frac{\overline{a_{j}}b_{j}}{\overline{\gamma_{j}}} > 0$ for $j = 1, 2$. 
By Lemma~$6$, the origin does not belong to $S_{1}(F_{t})$. 
Thus a~deformation $F_{t}(\zz)$ of $f(\zz)\overline{g}(\zz)$ satisfies the~condition $(3)$ and 
$F_{t}(\zz) \neq 0$ on $S_{1}(F_{t})$. 
%Since $\{ \Psi(\zz, \alpha) = 0 \}$ and $\{ f(\zz)\overline{g}(\zz) - \overline{\alpha}\overline{f}(\zz)g(\zz) = 0\}$ 
%have finitely many branches, 
%setting $f(\zz) = c^{1}_{f}z_{1}^{pm} + c^{2}_{f}z_{2}^{qm} + z_{1}^{p}z_{2}^{q}f'(\zz)$ 
%and $g(\zz) = c^{1}_{g}z_{1}^{pm} + c^{2}_{g}z_{2}^{qn} + z_{1}^{p}z_{2}^{q}g'(\zz)$, 
%we apply the argument used in the last paragraph in the proof of Lemma~\ref{l7}. 
%we can choose coefficients $\gamma_{1}$ and $\gamma_{2}$ of $h(\zz)$ such that 
%$\Re \frac{\overline{c^{j}_{f}}c^{j}_{g}}{\overline{\gamma_{j}}} > 0$ and, on $S_{1}(F_{t})$,  
%$\Phi(\zz, \alpha)$ and $f(\zz)\overline{g}(\zz) - \overline{\alpha}\overline{f}(\zz)g(\zz)$ 
%are non-zero, where $j = 1, 2$.  
\end{proof}

\begin{lemma}\label{l8}
Let $F_{t}(\zz)$ be a deformation of $f(\zz)\overline{g}(\zz)$ in Lemma \ref{l13}. 
$S_{1}(F_{t})$ are indefinite fold singularities. 
\end{lemma}

\begin{proof}
By Proposition $2$, $S_{1}(F_{t})$ is a union of the orbits of the $S^{1}$-action. 
So a connected component of $S_{1}(F_{t})$ can be represented by 
\[
(e^{iq\theta}z_{1}, e^{ip\theta}z_{2}), \ \ \theta \in [0, 2\pi]. 
\]
We first show that the differential of $F_{t}|_{S_{1}(F_{t})} : S_{1}(F_{t}) \rightarrow \Bbb{R}^{2}$ is non-zero. 
On a connected component of $S_{1}(F_{t})$, the map $F_{t}$ has the following form: 
\[
F_{t}(e^{iq\theta}z_{1}, e^{ip\theta}z_{2}) = e^{ipq(m-n)\theta}F_{t}(z_{1}, z_{2}). 
\]
Thus the differential of $F_{t}$ satisfies  
\[
\frac{d F_{t}}{d \theta}(e^{iq\theta}z_{1}, e^{ip\theta}z_{2}) = ipq(m-n)e^{ipq(m-n)\theta}F_{t}(z_{1}, z_{2}). 
\]
%Recall the calculation of Lemma $4$. 
%We can choose $h(\zz)$ such that the intersection of the following equations is only the origin 
%\begin{equation*}
%\begin{split}
%f(\zz)\overline{g}(\zz) + th(\zz) &= 0, \\
%\overline{\frac{\partial f}{\partial z_{1}}}(\zz)g(\zz) + t\overline{\frac{\partial h}{\partial z_{1}}}(\zz) &= 
%\alpha f(\zz)\overline{\frac{\partial g}{\partial z_{1}}}(\zz), \\
%\overline{\frac{\partial f}{\partial z_{2}}}(\zz)g(\zz) + t\overline{\frac{\partial h}{\partial z_{2}}}(\zz) &= 
%\alpha f(\zz)\overline{\frac{\partial g}{\partial z_{2}}}(\zz). 
%\end{split}
%\end{equation*}
%By the choice of $h(\zz)$,  
Since $F_{t}(\zz)$ does not vanish on $S_{1}(F_{t})$, 
the differential does not vanish on $S_{1}(F_{t})$. 
Thus any point of~$S_{1}(F_{t})$ is a fold singularity.

%Let $S$ be a connected component of $S_{1}(F_{t})$ and 
%$\zz(u) = (z_{1} + u, z_{2})$ be a path on $\Bbb{C}^{2}$, where $(z_{1}, z_{2}) \in S$ and $0 \leq u << 1$. 
Next we calculate the differential of $\lvert F_{t} \rvert^{2}$. 
Let $S$ be a connected component of $S_{1}(F_{t})$ and 
$W$ be a sufficiently small neighborhood of $S$ in $\Bbb{R}^{4}$. 
Assume that $S$ is the set of definite fold singularities. 
Since $\lvert F_{t} \rvert^{2}$ is constant on $S$, 
$\lvert F_{t}(\ww) \rvert^{2}$ is the maximal value or the minimum value for any $\ww \in S$. 
%So $\lvert F_{t}(\zz(u)) \rvert^{2}$ is monotone increasing or 
%monotone decreasing for $0 < u << 1$. 
Suppose that $\lvert F_{t}(\ww) \rvert^{2}$ is the minimum value for any $\ww \in S$. 
Since $S$ is a connected component of $S_{1}(F_{t})$, 
the differential 
$\frac{\partial F_{t}}{\partial z_{1}}$ or $\frac{\partial F_{t}}{\partial z_{2}}$ does not vanish on $S$. 
Assume that $\frac{\partial F_{t}}{\partial z_{1}} \neq 0$ on $W$. 
Let $\zz(u) = (z_{1} + u, z_{2})$ be a curve on $\Bbb{C}^{2}$, where $(z_{1}, z_{2}) \in S$ and $0 \leq u << 1$. 
By the definition of $F_{t}(\zz) = f(\zz)\overline{g}(\zz) + th(\zz)$, 
$\frac{\partial F_{t}}{\partial \overline{z}_{1}} = \overline{\frac{\partial \overline{F_{t}}}{\partial z_{1}}}$ and 
$\frac{\partial \overline{F_{t}}}{\partial \overline{z}_{1}} = \overline{\frac{\partial F_{t}}{\partial z_{1}}}$. 
Then we have 
\begin{equation*}
\begin{split}
\frac{\partial \lvert F_{t} \rvert^{2}}{\partial \overline{z}_{1}} &= 
\frac{\partial F_{t}}{\partial \overline{z}_{1}}\overline{F_{t}} + F_{t}\frac{\partial \overline{F_{t}}}{\partial \overline{z}_{1}} \\
&= \overline{\frac{\partial \overline{F_{t}}}{\partial z_{1}}}\overline{F_{t}} + F_{t}\overline{\frac{\partial F_{t}}{\partial z_{1}}} \\
&= \overline{\frac{\partial \lvert F_{t} \rvert^{2}}{\partial z_{1}}}. 
\end{split}
\end{equation*}
Thus the differential of $\lvert F_{t}(\zz(u)) \rvert^{2}$ satisfies 
\begin{equation*}
\begin{split}
\frac{d \lvert F_{t}(\zz(u)) \rvert^{2}}{d u} 
&= \frac{\partial \lvert F_{t} \rvert^{2}}{\partial z_{1}}\frac{d z_{1}}{d u} + 
\frac{\partial \lvert F_{t} \rvert^{2}}{\partial \overline{z}_{1}}\frac{d \overline{z}_{1}}{d u} \\
&= \frac{\partial \lvert F_{t} \rvert^{2}}{\partial z_{1}}\frac{d z_{1}}{d u} + 
\overline{\frac{\partial \lvert F_{t} \rvert^{2}}{\partial z_{1}}}\frac{d \overline{z}_{1}}{d u} \\
&= 2\Re \frac{\partial \lvert F_{t} \rvert^{2}}{\partial z_{1}} \geq 0. \\
\end{split}
\end{equation*}
So there exists a neighborhood of $(z_{1}, z_{2})$ such that the real part of 
$\frac{\partial \lvert F_{t}(\zz) \rvert^{2}}{\partial z_{1}}$ is non-negative. 
%$\lvert F_{t}(w) \rvert \leq \lvert F_{t}(\zz) \rvert$ or 
%$\lvert F_{t}(w) \rvert \geq \lvert F_{t}(\zz) \rvert$
%for any $w \in W$ and $\zz \in S$. 
%Since the differential $\frac{d \lvert F_{t}(\zz(0) \rvert^{2}}{d u}$ 
%is equal to $0$ 
%and the number of the critical points of $\lvert F_{t}(\zz(u)) \rvert^{2}$ is finite, 
%Then we use the $S^{1}$-action on $\Bbb{C}^{2}$, 
%\[
%s \circ \zz(u) := (s^{q}(z_{1} + u), s^{p}z_{2}). 
%\]
Since $\frac{\partial \lvert F_{t}(\zz) \rvert^{2}}{\partial z_{1}}$ is a polar weighted homogeneous mixed polynomial, 
$\frac{\partial \lvert F_{t}(s \circ \zz) \rvert^{2}}{\partial z_{1}} = 
\overline{s}^{q}\frac{\partial \lvert F_{t}(\zz) \rvert^{2}}{\partial z_{1}}$. 
We take the complex number $s$ such that $\lvert s \rvert = 1$, on a neighborhood of $s \circ (z_{1}, z_{2})$, and 
\[
2\Re \frac{\partial \lvert F_{t} \rvert^{2}}{\partial z_{1}} \leq 0. 
\]
%$\overline{s}^{q} = -1$, 
%$\lvert F_{t}(s \circ \zz(u)) \rvert^{2}$ is monotone decrease for $0 < u << 1$. 
%If $r$ is greater than $1$ $(\text{resp.}\ r \ \text{is less than} \ 1)$, $\lvert f\overline{g}~+~h \rvert^{2}(r~\circ~\zz) > \lvert f\overline{g} + h \rvert^{2}(\zz) 
%(\text{resp.}\ \lvert f\overline{g} + h \rvert^{2}(r \circ \zz) < \lvert f\overline{g} + h \rvert^{2}(\zz))$. 
%$\lvert F_{t}(w) \rvert \leq \lvert F_{t}(\zz) \rvert$ or 
%$\lvert F_{t}(w) \rvert \geq \lvert F_{t}(\zz) \rvert$
%for any $w \in W$ and $\zz \in S$. 
%the absolute value of $F_{t}$ is monotone increase or 
%monotone decrease on a neighborhood of 
%$S_{1}(F_{t})$. 
%Then $\lvert F_{t}(s \circ \zz(0)) \rvert^{2}$ is the maximum value or the minimum value for any $s \in S^{1}$. 
Then there exists a curve $\zz'(u)$ such that $\zz'(0) = s \circ (z_{1}, z_{2}) \in S$ and 
$\lvert F_{t} \rvert^{2}$ is monotone decreasing on~$\zz'(u)$. 
This is a contradiction. 
If $\frac{\partial F_{t}}{\partial z_{2}}$ does not vanish on $W$, 
the proof is analogous in case~$\frac{\partial F_{t}}{\partial z_{1}} \neq 0$. 
%So $S_{1}(F_{t})$ is the set of indefinite fold singularities. 
If $\lvert F_{t}(\ww) \rvert^{2}$ is the maximal value, 
we can calculate the~differential of $\lvert F_{t} \rvert^{2}$ by using the same method 
of the above case. 
Thus we show that $S_{1}(F_{t})$ is the set of indefinite fold singularities. 
\end{proof}

\subsection{Case $(2)$.}
In this subsection, $g(\zz)$ is equal to $\beta_{1} z_{1} + \beta_{2} z_{2}$. 
Since we study $f(\zz)\overline{g}(\zz)$, we may assume that $g(\zz) = z_{1} + \beta z_{2}$. 
We study the following deformation $F_{t}(\zz)$ of $f(\zz)\overline{(z_{1} + \beta z_{2})}$: 
\[
F_{t}(\zz) := f(\zz)\overline{(z_{1} + \beta z_{2})} + th(\zz), 
%(z_{1}^{m}\overline{z_{1}} + z_{1}^{m-1} + \gamma z_{2}^{m-1}). 
\]
where $h(\zz) = z_{1}^{m}\overline{z_{1}} + z_{1}^{m-1} + \gamma z_{2}^{m-1}$. 
We study the rank of $H(F_{t})$ and the differential of $F_{t}|_{S_{1}(F_{t})}$. 
\begin{lemma}\label{l9}
%Let $F_{t}(\zz)$ be the above deformation of $f(\zz)\overline{g}(\zz)$. 
There exists a coefficient $\gamma$ such that 
the singularities of $F_{t}(\zz)$ are indefinite fold 
singularities except for the origin.  
\end{lemma}
\begin{proof}
If $\ww$ is a singularity of $F_{t}(\zz)$, 
$\ww$ satisfies the following equation: 
\[
\Phi'(\zz, \alpha) := \overline{\frac{\partial f}{\partial z_{1}}}(\zz)\overline{\frac{\partial h}{\partial z_{2}}}(\zz)g(\zz) + 
\Bigl(\alpha\overline{\beta}f(\zz) - \overline{\frac{\partial f}{\partial z_{2}}}(\zz)g(\zz) \Bigr)\overline{\frac{\partial h}{\partial z_{1}}}(\zz) - 
\alpha(f(\zz) + z_{1}^{m})\overline{\frac{\partial h}{\partial z_{2}}}(\zz) = 0. 
\]
Since $H(F_{0}) \equiv 0$, the determinant of $H(F_{t})$ has the form: 
\[
t^{2}m^{2}\overline{\beta}^{2}z_{1}^{2m-2}\Bigr(\frac{\partial f}{\partial z_{2}}\Bigl)^{2}. 
%\overline{\frac{\partial g}{\partial z_{2}}}. 
\]
Hence $\det H(F_{t})$ is a complex polynomial which is not identically zero. 
So $\{ \det H(F_{t}) = 0 \}$ has finitely many branches. 
We can choose $h(\zz)$ such that 
the intersection of $\{ \det H(F_{t}) = 0 \}$ and $\{ \Phi'(\zz, \alpha) = 0 \}$ is the origin. 
%\[
%\Phi(\zz, \alpha):= \biggl(\overline{\frac{\partial f}{\partial z_{1}}\frac{\partial h}{\partial z_{2}} - \frac{\partial f}{\partial z_{2}}\frac{\partial h}{\partial z_{1}}}\biggr)g(\zz) - 
%\alpha \biggl(\overline{\frac{\partial g}{\partial z_{1}}\frac{\partial h}{\partial z_{2}} - \frac{\partial g}{\partial z_{2}}\frac{\partial h}{\partial z_{1}}}\biggr)f(\zz) = 0. 
%\]
Since $F_{t}(\zz)$ is also a polar weighted homogeneous mixed polynomial, 
we can show that the singularities of $F_{t}(\zz)$ are indefinite fold singularities, 
by using the same way as Lemma \ref{l8}. 
\end{proof}

%\noindent
%{\it \sc{Proof of Theorem~\ref{thm1}.}\;}
%The theorem follows from Lemma $4$, Lemma $5$ and Lemma $7$.
\begin{proof}[Proof of Theorem 1]
%We choose  coefficients of 
%$h(\zz)$ which satisfies Lemma \ref{l13}, Lemma~\ref{l8} and Lemma~\ref{l9}.
The singularities of the deformation $F_{t}(\zz)$ of $f(\zz)\overline{g}(\zz)$ except for the~origin 
are indefinite fold singularities by Lemma \ref{l13}, Lemma~\ref{l8} and Lemma~\ref{l9}.  
\end{proof}
%\qed

If the origin of $\Bbb{C}^{2}$ is a singularity of $F_{t}(\zz)$, 
$F_{t}^{-1}(0) \cap S^{3}_{\varepsilon_{t}}$ is an oriented link in $S^{3}_{\varepsilon_{t}}$ 
for a~sufficiently small $\varepsilon_{t}$. 
We study the topology of $F_{t}^{-1}(0) \cap S^{3}_{\varepsilon_{t}}$. 
%for a sufficiently small $\varepsilon_{t}$. 

\begin{lemma}\label{l10}
%The number of the connected components of 
The link $F_{t}^{-1}(0) \cap S^{3}_{\varepsilon_{t}}$ is 
%$m - n$ and  
%each connected component of $F_{t}^{-1}(0) \cap S^{3}_{\varepsilon_{t}}$ is a 
a $\bigl(p(m-n), q(m-n)\bigr)$-torus link. 
\end{lemma}

\begin{proof}
The deformation $F_{t}(\zz)$ of $f(\zz)\overline{g}(\zz)$ is a convenient non-degenerate mixed polynomial 
in sense of \cite{O2}. 
Let $\Delta$ be the compact face of the Newton boundary of $F_{t}(\zz)$. 
Then the~face function $F_{t \Delta}(\zz)$ is $t(\gamma_{1}z_{1}^{p(m-n)} + \gamma_{2} z_{2}^{q(m-n)})$.   
In \cite[Theorem 43]{O2}, the number of the~connected components of~$F_{t}^{-1}(0) \cap S^{3}_{\varepsilon_{t}}$ 
is equal to that of $F_{t \Delta}^{-1}(0) \cap S^{3}_{\varepsilon_{t}}$ where $0 < \varepsilon_{t} << 1$. 
Since $F_{t \Delta}(\zz) = t(\gamma_{1}z_{1}^{p(m-n)} + \gamma_{2} z_{2}^{q(m-n)})$ has $m - n$ irreducible components, 
the number of link components of $F_{t}^{-1}(0) \cap S^{3}_{\varepsilon_{t}}$ is equal to $m - n$.
By the choice of $h(\zz)$, 
$F_{t \Delta}$ is a polar weighted homogeneous polynomial. 
So $F_{t \Delta}^{-1}(0)$ is an invariant set of the~$S^{1}$-action. 
In \cite{EN}, the connected component of $F_{t \Delta}^{-1}(0) \cap S^{3}_{\varepsilon_{t}}$ is isotopic to a $(p, q)$-torus knot
whose orientation coincides with that of the $S^{1}$-action and 
the linking numbers of components of $F_{t}^{-1}(0) \cap S^{3}_{\varepsilon_{t}}$ are equal to $pq$. 
\end{proof}

\section{Proof of Theorem~\ref{thm2}}
Let $F_{t}(\zz)$ be a deformation of $f(\zz)\overline{g}(\zz)$ in Theorem $1$. 
In this section, we study the~deformation of~$F_{t}(\zz)$: 
\[
F_{t, s}(\zz) := f(\zz)\overline{g}(\zz) + th(\zz) + s\ell(\zz), 
\]
where $\ell(\zz) = c_{1}z_{1} + c_{2}z_{2}, c_{1}, c_{2} \in \Bbb{C} \setminus \{ 0 \}$ and $0 < s << t << 1$. 
Suppose that $c_{1}$ and $c_{2}$ satisfy 
\renewcommand{\theenumi}{\roman{enumi}}
\begin{enumerate}
\item
$\{td_{h}h(\zz) + s(qc_{1}z_{1} + pc_{2}z_{2}) = 0 \}$ and $\{ f(\zz)\overline{g}(\zz) = 0 \}$ have no common branches, 
\item
$\{(d_{h} - q)qc_{1}z_{1} + (d_{h} - p)pc_{2}z_{2} = 0 \}$ and $\{ f(\zz)\overline{g}(\zz) = 0 \}$ have no common branches, 
\end{enumerate}
where $d_{h} = pq(m - n)$ and $0 < s << t << 1$. 
%Let $\zz$ belong to $S_{2}(F_{t, s})$. Then $\zz$~satisfies $f(\zz)\overline{\frac{\partial g}{\partial z_{j}}(\zz)} = 0$ and 
%$\frac{\partial f}{\partial z_{j}}\overline{g}(\zz) + t\frac{\partial h}{\partial z_{j}} + sa_{j} = 0$ for $j = 1, 2$. 
%By using the Euler equality, $tpq(m - n)h(\zz) + s(qa_{1}z_{1} + pa_{2}z_{2}) = 0$. 
%If $f(\zz)\overline{g}(\zz)$ and $qa_{1}z_{1} + pa_{2}z_{2}$ have common branches, $\zz$~is equal to the origin. 
%By the definition of $F_{t, s}$, the origin is not singularity of $F_{t, s}$. This is a contradiction. 
%So we assume that $f(\zz)\overline{g}(\zz)$ and $qa_{1}z_{1} + pa_{2}z_{2}$ have no common branches. 
The mixed Hessian $H(F_{t, s})$ of~$F_{t,s}$ is equal to~$H(F_{t, 0})$. 
To prove Theorem $2$, 
we first show that non-isolated singularities of $F_{t, s}(\zz)$ are indefinite fold singularities. 

\begin{lemma}\label{l11}
%Let $\zz$ be a singularity of $F_{t, s}$ such that $\zz \in S_{1}(F_{t, s})$. 
There exist $c_{1}$ and $c_{2}$ such that 
any point of $S_{1}(F_{t, s})$ is an indefinite fold singularity, where $0 <~s <<~t<< 1$. 
\end{lemma}
\begin{proof}
In the proof of Theorem~$1$, we proved that $\{ \det H(F_{t, 0}) = 0 \} \cap S_{1}(F_{t,0}) = \emptyset$ and  
the differential of $F_{t, 0}|_{S_{1}(F_{t, 0})}$ is non-zero. 
So there exists a neighborhood $U_{F_{t, 0}}$ of $S_{1}(F_{t, 0})$ such that 
\begin{equation*}
\begin{split}
\{ \det H(F_{t, 0}) = 0 \} \cap U_{F_{t, 0}} = \emptyset, \\
\frac{d}{d x_{1}^{0}}F_{t, 0}(\ww) \neq 0,  
\end{split}
\end{equation*}
where $x_{1}^{0}$ is a coordinate of $S_{1}(F_{t, 0})$ in $\Bbb{R}^{4}$ and $\ww \in S_{1}(F_{t, 0})$. 
%Let $\zz$ be a point of $\{ \det H(F_{t}) = 0 \}$ such that $\zz \neq \mathbf{o}$. 
We take non-zero complex numbers $c_{1}, c_{2}$ and sufficiently small positive real number $s_{0}$ such that 
$S_{1}(F_{t, s}) \subset U_{F_{t, 0}}$ for any $0 < s \leq s_{0}$. 
%Since $U \cap \{ \det H(F_{t, s}) = 0 \}$ is a compact set, 
Then  
the intersection of $S_{1}(F_{t, s})$ and $\{ \det H(F_{t, s}) = 0 \}$ is empty. 
Thus $j^{1}F_{t, s}$ is transversal to $S_{1}(\Bbb{R}^{4}, \Bbb{R}^{2})$ at~$S_{1}(F_{t, s})$. 
We check the differential of $F_{t, s} : S_{1}(F_{t, s}) \rightarrow \Bbb{R}^{2}$. 
Let $(x_{1}^{s}, \dots, x_{4}^{s})$ be a family of coordinates of $\Bbb{R}^4$, smoothly parametrized by $s$, 
such that $x_{1}^{s}$ is the coordinate of 
$S_{1}(F_{t, s})$. 
%Then there exists a $C^{\infty}$ map $\phi_{s} : (-\delta, \delta) \times U_{0} \rightarrow U_{s}, (s, x_{1}^{0}, \dots, x_{4}^{0}) \mapsto (x_{1}^{s}, \dots, x_{4}^{s})$, 
%where 
%$U_{s}$ is a neighborhood of $S_{1}(f_{t, s})$ and $0 < \delta << 1$. 
Then we have 
\[
\frac{dF_{t,s}}{dx_{1}^{s}} = \frac{\partial F_{t, 0}}{\partial x_{1}^{0}}\frac{\partial x_{1}^{0}}{\partial x_{1}^{s}} + \cdots + 
\frac{\partial F_{t, 0}}{\partial x_{4}^{0}}\frac{\partial x_{4}^{0}}{\partial x_{1}^{s}} 
+ s\Bigl(\frac{\partial \ell}{\partial x_{1}^{0}}\frac{\partial x_{1}^{0}}{\partial x_{1}^{s}} + \cdots  
+ \frac{\partial \ell}{\partial x_{4}^{0}}\frac{\partial x_{4}^{0}}{\partial x_{1}^{s}} \Bigl). 
\]
Since $\frac{\partial F_{t, 0}}{\partial x_{1}^{0}}$ is non-zero on $U_{F_{t, 0}}$, 
$\frac{dF_{t, s}}{dx_{1}^{s}}$ is non-zero for $0 < s << 1$. 
Thus any point of $S_{1}(F_{t, s})$ is a fold singularity. 

By changing coordinates of $\Bbb{R}^{4}$, we may assume that $F_{t, s} : \Bbb{R}^{4} \rightarrow \Bbb{R}^{2}$ 
has the following form: 
\[
F_{t, s} = \bigl(x_{1}^{s}, I_{t, s}(x_{2}^{s}, x_{3}^{s}, x_{4}^{s})\bigr) 
\] 
on $U_{F_{t, 0}}$. 
%By using coordinates $(x_{1}^{0}, \dots, x_{4}^{0})$ of $\Bbb{R}^{4}$, 
Since $I_{t, 0}$ is $\lvert F_{t, 0}\lvert$, 
we set $I_{t, s} = \lvert F_{t, 0}(x_{1}^{0}, \dots, x_{4}^{0})\lvert + sI'_{t, s}(x_{1}^{0}, \dots, x_{4}^{0})$. 
Since $S_{1}(F_{t, 0})$ are indefinite fold singularities, 
%on $U_{F_{t, 0}}$,  
%$\lvert F_{t, 0} \rvert = -(x_{2}^{0})^{2} + (x_{3}^{0})^{2} + e(x_{4}^{0})^{2}$, 
by choosing suitable coordinates $(x_{2}^{0}, x_{3}^{0}, x_{4}^{0})$, 
%Then $I_{t, s}$ satisfies  
we may assume that 
\begin{equation*}
\begin{split}
%\lvert F_{t, 0} \rvert + s\iota_{t, s} 
I_{t, s} &= 
-(x_{2}^{0})^{2} + (x_{3}^{0})^{2} + e(x_{4}^{0})^{2} + sI'_{t, s}(x_{1}^{0}, \dots, x_{4}^{0}), \\
\frac{\partial I_{t, s}}{\partial x_{2}^{0}} &= 
-2x_{2}^{0} + s\frac{\partial I'_{t, s}}{\partial x_{2}^{0}}, \\
\end{split}
\end{equation*}
where $e = \pm 1$. 
Let $(\iota_{1}^{s}, \iota_{2}^{s}, \iota_{3}^{s}, \iota_{4}^{s})$ 
be a point of $S_{1}(F_{t, s})$. 
Then we have 
\[
-2\iota_{2}^{s} + s\frac{\partial I'_{t, s}}{\partial x_{2}^{0}}(\iota_{1}^{s}, \iota_{2}^{s}, \iota_{3}^{s}, \iota_{4}^{s}) = 0. 
\]
We first fix $x_{1}^{0}, x_{3}^{0}$ and $x_{4}^{0}$, 
i.e., $x_{1}^{0} = \iota_{1}^{s}, x_{3}^{0} = \iota_{3}^{s}$ and $x_{4}^{0} = \iota_{4}^{s}$. 
Since $s$ is sufficiently small, 
$\frac{\partial I_{t, s}}{\partial x_{2}^{0}}$ satisfies 
\begin{equation*}
\begin{split}
\frac{\partial I_{t, s}}{\partial x_{2}^{0}} 
&= -2x_{2}^{0} + s\frac{\partial I'_{t, s}}{\partial x_{2}^{0}}(\iota_{1}^{s}, x_{2}^{0}, \iota_{3}^{s}, \iota_{4}^{s})  \\
&= -2x_{2}^{0} + s\frac{\partial I'_{t, s}}{\partial x_{2}^{0}}(\iota_{1}^{s}, x_{2}^{0}, \iota_{3}^{s}, \iota_{4}^{s}) 
+ 2\iota_{2}^{0} - s\frac{\partial I'_{t, s}}{\partial x_{2}^{0}}(\iota_{1}^{s}, \iota_{2}^{s}, \iota_{3}^{s}, \iota_{4}^{s}) \\
&= -2\bigl(x_{2}^{0} - \iota_{2}^{s} \bigr) + 
s\biggl(\frac{\partial I'_{t, s}}{\partial x_{2}^{0}}(\iota_{1}^{s}, x_{2}^{0}, \iota_{3}^{s}, \iota_{4}^{s}) - 
\frac{\partial I'_{t, s}}{\partial x_{2}^{0}}(\iota_{1}^{s}, \iota_{2}^{s}, \iota_{3}^{s}, \iota_{4}^{s}) \biggr) \\
&= \bigl(x_{2}^{0} - \iota_{2}^{s} \bigr)\Biggl(-2 + s\frac{\frac{\partial I'_{t, s}}{\partial x_{2}^{0}}(\iota_{1}^{s}, x_{2}^{0}, \iota_{3}^{s}, \iota_{4}^{s}) - \frac{\partial I'_{t, s}}{\partial x_{2}^{0}}(\iota_{1}^{s}, \iota_{2}^{s}, \iota_{3}^{s}, \iota_{4}^{s})}{x_{2}^{0} - \iota_{2}^{s}} \Biggr) < 0. 
\end{split}
\end{equation*}
Thus there exists 
a~curve $\zz_{s}(u)$ on $U_{F_{t, 0}}$ such that 
$\zz_{s}(0) \in S_{1}(F_{t, s})$ and 
$I_{t, s}$ is monotone decreasing on $\zz_{s}(u)$. 
Next we fix $x_{1}^{0}, x_{2}^{0}$ and $x_{4}^{0}$. Then we can show that there exists  
a~curve $\zz'_{s}(u)$ on $U_{F_{t, 0}}$ such that 
$\zz'_{s}(0) \in S_{1}(F_{t, s})$ and 
$I_{t, s}$ is monotone increasing on $\zz'_{s}(u)$. 
So $S_{1}(F_{t, s})$ is the set of indefinite fold singularities. 
\end{proof}

Next we consider isolated singularities of $F_{t, s}(\zz)$. 
Then these singularities belong to $S_{2}(F_{t,s})$. 
We study the topological types of the links at each point of $S_{2}(F_{t,s})$. 

\begin{lemma}\label{l12}
Let $F_{t, s}(\zz)$ be a deformation of $F_{t}(\zz)$ in Lemma \ref{l11}. 
Then $S_{2}(F_{t, s})$ is the set of finite mixed Morse singularities. 
\end{lemma}
\begin{proof}
If $\ww = (z_{1}, z_{2})$ belongs to $S_{2}(F_{t,s})$, $f(\ww)\overline{g}(\ww) = 0$ 
and $td_{h}h(\ww) + s(qc_{1}z_{1} + pc_{2}z_{2}) = 0$ 
by Proposition $1$ and equation $(1)$. 
Since $c_{1}$ and $c_{2}$ satisfy the condition $(\text{i})$, 
the number of $S_{2}(F_{t, s})$ is finite. 
The~determinant of $H(F_{t, s})$ is equal to 
\[ H(F_{t, s}) =    \left(
        \begin{array}{@{\,}cccccccc@{\,}}
         \frac{\partial^{2}f}{\partial z_{1}\partial z_{1}}\overline{g} + 
         t\frac{\partial^{2}h}{\partial z_{1}\partial z_{1}}
         & \frac{\partial^{2}f}{\partial z_{2}\partial z_{1}}\overline{g} 
         & \frac{\partial f}{\partial z_{1}}\overline{\frac{\partial g}{\partial z_{1}}} 
         & \frac{\partial f}{\partial z_{1}}\overline{\frac{\partial g}{\partial z_{2}}} \\ 
         \frac{\partial^{2}f}{\partial z_{1}\partial z_{2}}\overline{g}  
         &  \frac{\partial^{2}f}{\partial z_{2}\partial z_{2}}\overline{g} + 
         t\frac{\partial^{2}h}{\partial z_{2}\partial z_{2}}  
         & \frac{\partial f}{\partial z_{2}}\overline{\frac{\partial g}{\partial z_{1}}} 
         & \frac{\partial f}{\partial z_{2}}\overline{\frac{\partial g}{\partial z_{2}}} \\ 
         \frac{\partial f}{\partial z_{1}}\overline{\frac{\partial g}{\partial z_{1}}} 
         & \frac{\partial f}{\partial z_{2}}\overline{\frac{\partial g}{\partial z_{1}}} 
         & 0 
         & 0 \\ 
         \frac{\partial f}{\partial z_{1}}\overline{\frac{\partial g}{\partial z_{2}}} 
         & \frac{\partial f}{\partial z_{2}}\overline{\frac{\partial g}{\partial z_{2}}} 
         & 0 
         & 0  
        \end{array}
       \right).  \]
Assume that $\frac{\partial f}{\partial z_{j}}(\ww)\overline{\frac{\partial g}{\partial z_{k}}}(\ww) = 0$ for any $j, k \in \{ 1,2 \}$. 
By using equation $(1)$, 
$f(\ww)\overline{\frac{\partial g}{\partial z_{j}}}(\ww)~= \frac{\partial f}{\partial z_{j}}(\ww)\overline{g}(\ww) = 0$. 
Hence $\ww$ is a singularity of $f(\zz)\overline{g}(\zz)$ by Proposition~$1$. 
Since $f(\zz)\overline{g}(\zz)$ has an~isolated singularity at the origin, $\ww$ is the origin. 
Then we have 
\begin{equation*}
\begin{split}
\frac{\partial F_{t, s}}{\partial z_{j}}(\ww) = \frac{\partial F_{t, s}}{\partial z_{j}}(\mathbf{o}) = 
\frac{\partial f}{\partial z_{j}}(\mathbf{o})\overline{g}(\mathbf{o}) + 
t\frac{\partial h}{\partial z_{j}}(\mathbf{o}) + sc_{j} = sc_{j} \neq 0, 
\end{split}
\end{equation*}
where $s \neq 0$. The origin $\mathbf{o}$ does not belong to $S_{2}(F_{t, s})$. This is a contradiction. 
So there exist $j, k \in \{1, 2 \}$ such that 
$\frac{\partial f}{\partial z_{j}}(\ww)\overline{\frac{\partial g}{\partial z_{k}}}(\ww)$ 
is non-zero at $\ww \in S_{2}(F_{t, s})$. 
Assume $\frac{\partial f}{\partial z_{2}}(\ww)\overline{\frac{\partial g}{\partial z_{1}}}(\ww)$ is non-zero. 
We calculate $H(F_{t, s})$ by using equation $(1)$. 
Since $c_{1}$ and $c_{2}$ satisfy the condition $(\text{ii})$, the Hessian  $H(F_{t, s})$ is congruent to 
\[ H(F_{t, s}) \cong    \left(
        \begin{array}{@{\,}cccccccc@{\,}}
-s\{(d_{h}-q)qc_{1}z_{1} + (d_{h}-p)pc_{2}z_{2} \} & 0 & 0 & 0 \\
0 & 0 & \frac{\partial f}{\partial z_{2}}\overline{\frac{\partial g}{\partial z_{1}}} & 0 \\ 
0 & \frac{\partial f}{\partial z_{2}}\overline{\frac{\partial g}{\partial z_{1}}} & 0 & 0 \\
0 & 0 & 0 & 0
\end{array}
       \right) 
       \cong    \left(
        \begin{array}{@{\,}cccccccc@{\,}}
1 & 0 & 0 & 0 \\
0 & 0 & 1 & 0 \\ 
0 & 1 & 0 & 0 \\
0 & 0 & 0 & 0
\end{array}
       \right) 
        \]
at $\ww \in S_{2}(F_{t, s})$, where $d_{h} = pq(m - n)$. 
By the same argument, we can check that 
if either $\frac{\partial f}{\partial z_{1}}(\ww)\overline{\frac{\partial g}{\partial z_{1}}}(\ww)$, 
$\frac{\partial f}{\partial z_{1}}(\ww)\overline{\frac{\partial g}{\partial z_{2}}}(\ww)$ or 
$\frac{\partial f}{\partial z_{2}}(\ww)\overline{\frac{\partial g}{\partial z_{2}}}(\ww)$ is non-zero, 
$H(F_{t, s})$ is congruent to the~above right-hand matrix. 
We change the coordinates of $\Bbb{C}^{2}$ such that, at the singularity of $F_{t, s}(\zz)$, 
the mixed Hessian $H(F_{t, s})$ is equal to 
\[
\left(
        \begin{array}{@{\,}cccccccc@{\,}}
1 & 0 & 0 & 0 \\
0 & 0 & 1 & 0 \\ 
0 & 1 & 0 & 0 \\
0 & 0 & 0 & 0
\end{array}
       \right). 
\]
We identify $\Bbb{C}^{2}$ with $\Bbb{R}^{4}$. 
Then $H_{\Bbb{R}}(\Re F_{t,s}) + iH_{\Bbb{R}}(\Im F_{t,s})$ has the following form: 
\begin{equation*}
\begin{split}
 H_{\Bbb{R}}(\Re F_{t,s}) + iH_{\Bbb{R}}(\Im F_{t,s}) &= 
\left(
        \begin{array}{@{\,}cccccccc@{\,}}
1 & 0 & 1 & 0 \\
0 & 1 & 0 & 1 \\ 
i & 0 & -i & 0 \\
0 & i & 0 & -i
\end{array}
       \right)
\left(
        \begin{array}{@{\,}cccccccc@{\,}}
1 & 0 & 0 & 0 \\
0 & 0 & 1 & 0 \\ 
0 & 1 & 0 & 0 \\
0 & 0 & 0 & 0
\end{array}
       \right)
       \left(
        \begin{array}{@{\,}cccccccc@{\,}}
1 & 0 & i & 0 \\
0 & 1 & 0 & i \\ 
1 & 0 & -i & 0 \\
0 & 1 & 0 & -i
\end{array}
       \right) \\
       &= 
       \left(
        \begin{array}{@{\,}cccccccc@{\,}}
1 & 1 & i & i \\
1 & 0 & -i & 0 \\ 
i & -i & -1 & 1 \\
i & 0 & 1 & 0
\end{array}
       \right) \\
       &= 
       \left(
        \begin{array}{@{\,}cccccccc@{\,}}
1 & 1 & 0 & 0 \\
1 & 0 & 0 & 0 \\ 
0 & 0 & -1 & 1 \\
0 & 0 & 1 & 0
\end{array}
       \right)
+ i\left(
        \begin{array}{@{\,}cccccccc@{\,}}
0 & 0 & 1 & 1 \\
0 & 0 & -1 & 0 \\ 
1 & -1 & 0 & 0 \\
1 & 0 & 0 & 0
\end{array}
       \right).
\end{split}
\end{equation*}

Since $H_{\Bbb{R}}(\Re F_{t, s})$ and $H_{\Bbb{R}}(\Im F_{t, s})$ are regular matrices, 
$\Re F_{t, s}$ and $\Im F_{t, s}$ are Morse functions. 
Put 
$R = \left(
\begin{array}{@{\,}cccccccc@{\,}}
        1 & -1 & 0 & 0 \\
        0 & 1 & 0 & 0 \\
        0 & 0 & 1 & 1 \\
        0 & 0 & 0 & 1 
\end{array}
 \right)$, we have 
\begin{equation*}
\begin{split}
 {}^{t}RH_{\Bbb{R}}(\Re F_{t, s})R &= 
 \left(
        \begin{array}{@{\,}cccccccc@{\,}}
        1 & 0 & 0 & 0 \\
        -1 & 1 & 0 & 0 \\
        0 & 0 & 1 & 0 \\
        0 & 0 & 1 & 1 
\end{array}
       \right)
\left(
        \begin{array}{@{\,}cccccccc@{\,}}
1 & 1 & 0 & 0 \\
1 & 0 & 0 & 0 \\ 
0 & 0 & -1 & 1 \\
0 & 0 & 1 & 0
\end{array}
       \right)
       \left(
        \begin{array}{@{\,}cccccccc@{\,}}
        1 & -1 & 0 & 0 \\
        0 & 1 & 0 & 0 \\
        0 & 0 & 1 & 1 \\
        0 & 0 & 0 & 1 
\end{array}
       \right) \\
       &= 
       \left(
        \begin{array}{@{\,}cccccccc@{\,}}
        1 & 0 & 0 & 0 \\
        0 & -1 & 0 & 0 \\
        0 & 0 & -1 & 0 \\
        0 & 0 & 0 & 1 
\end{array}
       \right). 
\end{split}
\end{equation*}
Hence there exist the coordinates of $\Bbb{R}^{4}$ such that $\Re F_{t, s}$ has the following form: 
\[
\Re F_{t,s} = x_{1}^{2} - y_{1}^{2} - x_{2}^{2} + y_{2}^{2}. 
\]
On the other hand, the Hessian of $\Im F_{t, s}$ is congruent to 
\begin{equation*}
\begin{split}
 {}^{t}RH_{\Bbb{R}}(\Im F_{t, s})R &= 
\left(
        \begin{array}{@{\,}cccccccc@{\,}}
        1 & 0 & 0 & 0 \\
        -1 & 1 & 0 & 0 \\
        0 & 0 & 1 & 0 \\
        0 & 0 & 1 & 1 
\end{array}
       \right)
\left(
        \begin{array}{@{\,}cccccccc@{\,}}
0 & 0 & 1 & 1 \\
0 & 0 & -1 & 0 \\ 
1 & -1 & 0 & 0 \\
1 & 0 & 0 & 0
\end{array}
       \right)
       \left(
        \begin{array}{@{\,}cccccccc@{\,}}
        1 & -1 & 0 & 0 \\
        0 & 1 & 0 & 0 \\
        0 & 0 & 1 & 1 \\
        0 & 0 & 0 & 1 
\end{array}
       \right) \\
       &= 
       \left(
        \begin{array}{@{\,}cccccccc@{\,}}
        0 & 0 & 1 & 2 \\
        0 & 0 & -2 & -3 \\
        1 & -2 & -1 & 0 \\
        2 & -3 & 0 & 0 
\end{array}
       \right). 
\end{split}
\end{equation*}
So $\Im F_{t, s}$ is equal to $2(x_{1}y_{1} + 2x_{1}y_{2} -2x_{2}y_{1} -3x_{2}y_{2})$ on a neighborhood of $\ww$.

Put $z_{j} = x_{j} + iy_{j}$ for $j = 1, 2$. We calculate $F_{t, s}(\zz) = \Re F_{t, s} + i\Im F_{t, s}$ 
on a neighborhood of $\ww$: 
\begin{equation*}
\begin{split}
\Re F_{t, s} + i\Im F_{t, s} &= x_{1}^{2} - y_{1}^{2} - x_{2}^{2} + y_{2}^{2} + 2i(x_{1}y_{1} + 2x_{1}y_{2} -2x_{2}y_{1} -3x_{2}y_{2}) \\
&= (x_{1} + iy_{1})^{2} - (x_{2} + iy_{2})^{2} + 4i(x_{1}y_{2} - x_{2}y_{1} - x_{2}y_{2}) \\
&= z_{1}^{2} - z_{2}^{2} + 4i\Bigr(\frac{z_{1} + \overline{z}_{1}}{2}\frac{z_{2} - \overline{z}_{2}}{2i} 
- \frac{z_{2} + \overline{z}_{2}}{2}\frac{z_{1} - \overline{z}_{1}}{2i} - \frac{z_{2} + \overline{z}_{2}}{2}\frac{z_{2} - \overline{z}_{2}}{2i} \Bigr) \\
&= z_{1}^{2} - 2z_{2}^{2} - 2z_{1}\overline{z}_{2} + 2z_{2}\overline{z}_{1} + \overline{z}_{2}^{2}. 
\end{split}
\end{equation*}

We change coordinates of $\Bbb{C}^{2}$: 
\[
v_{1} = z_{1} - \overline{z}_{2}, v_{2} = z_{2}. 
\]
Then $F_{t, s}$ is equal to $v_{1}^{2} + 2\overline{v_{1}}v_{2}$ at $\ww$. 
So the algebraic set $\{ (v_{1}, v_{2}) \mid v_{1}^{2} + 2\overline{v_{1}}v_{2}~=~0~\}$ has two components: 
\[
\{ v_{1} = 0 \}, \ \ \{(v_{1}, v_{2}) = (2\tilde{r}e^{i\theta'}, -\tilde{r}e^{3i\theta'}) \mid 0 < \tilde{r}, 0 \leq \theta' \leq 2\pi \} . 
\]
We define the $S^{1}$-action on $\Bbb{C}^2$: 
\[
(v_{1}, v_{2}) \mapsto (\tilde{s}v_{1}, \tilde{s}^{3}v_{2}), \ \ \tilde{s} \in S^{1}. 
\]
Then the set of the zero points of $F_{t,s}(\zz)$ at $\ww$ is an invariant set of the $S^{1}$-action. 
So the~link of~$F_{t,s}(\zz)$ at $\ww$ is the Seifert link in \cite{EN}. 
Since two components of the link of $F_{t, s}$ are trivial knots and the absolute value of the linking number is $1$,  
$\ww$ defines a Hopf link as an unoriented link. 

Let $B^{4}_{\delta}$ be the $4$-dimensional ball such that $F_{t, 0}^{-1}(0) \cap \partial B^{4}_{\delta}$
is isotopic to   
$F_{t \Delta}^{-1}(0) \cap \partial B^{4}_{\delta}$ 
and the~intersection of $B^{4}_{\delta}$ and the singularities of $F_{t, s}(\zz)$ is equal to $S_{2}(F_{t, s})$, 
where $F_{t \Delta}(\zz)$ is the~face function of $F_{t}(\zz)$. 
The restricted map $F_{t, s} : B^{4}_{\delta} \rightarrow D^{2}$ is an unfolding of 
$F_{t\Delta}^{-1}(0) \cap~\partial B^{4}_{\delta}$ in the sense of \cite{NR1}. 
By Lemma~\ref{l10}, $F_{t\Delta}^{-1}(0) \cap \partial B^{4}_{\delta}$ is isotopic to 
the~$\bigl(p(m-n), q(m-n)\bigr)$-torus link 
whose orientations coincide with 
that of links of holomorphic functions. 
Then there exists an unfolding which has only positive Hopf links and the enhancement to 
the Milnor number is equal to $0$ \cite[Theorem 5.6.]{NR1}. 
Note that the~enhancement to the Milnor number is a homotopy invariant of fibered links. 
Assume that there exists singularities of $F_{t, s}(\zz)$ such that they define negative Hopf links. 
Then the enhancement to the Milnor number is positive \cite[Theorem 5.4.]{NR1}. 
The homotopy type of $F_{t,0}^{-1}(0) \cap \partial B^{4}_{\delta}$ is different from that of links 
of holomorphic functions. 
By Lemma \ref{l10}, this is a contradiction. 
Any point of $S_{2}(F_{t, s})$ defines a~positive Hopf link. 
Thus $\ww$ is a~mixed Morse singularity. 
\end{proof}

\begin{proof}[Proof of Theorem 2]
Let $\ell(\zz) = c_{1}z_{1} + c_{2}z_{2}$ be a linear function in Lemma \ref{l11}. 
Any point of $S_{1}(F_{t, s})$ is an indefinite fold singularity. 
By Lemma \ref{l12}, isolated singularities of $F_{t, s}(\zz)$ are mixed Morse singularities. 
Thus $F_{t, s}(\zz)$ is a mixed broken Lefschetz fibration. 
\end{proof}

\end{document}